\documentclass[11pt]{article}

\usepackage[top=1.0in,bottom=1.0in,left=1.5in,right=1.5in]{geometry}

\usepackage[dvipsnames]{xcolor}

\usepackage{mathtools}
\usepackage{tikz}

\usetikzlibrary{external}

\usepackage[small]{caption}
\usepackage{cite} 
\usepackage{setspace}
\usepackage{amsmath}
\usepackage{amsthm}
\usepackage{subfig}
\usepackage{amsfonts}
\usepackage{amssymb}
\usepackage{graphicx}
\usepackage{epsfig}
\usepackage{url}
\usepackage{todo}
\usepackage{bbm}

\usepackage{tikz}
\usepackage[dvipsnames]{xcolor}
\usetikzlibrary{cd}
\usepackage{pgfplots}
\usepackage{microtype}

\newtheorem{theorem}{Theorem}

\newtheorem{proposition}[theorem]{Proposition}

\newtheorem{lemma}{Lemma}

\newtheorem{remark}{Remark}
\newtheorem{example}{Example}
\newtheorem*{main}{Main Theorem}
\def\qed{ \rule{.08in}{.08in}}

\newcommand{\diag}{\operatorname{diag}}
\newcommand{\Diag}{\operatorname{Diag}}

\newcommand{\R}{\mathbb{R}}

\newcommand{\bP}{\mathbb{P}}

\newcommand{\U}{\mathrm{Uni}}
\newcommand{\ce}{\mathcal{E}}
\newcommand{\conv}{\operatorname{conv}}
\newcommand{\cX}{\mathcal{X}}

\newcommand{\cA}{\mathcal{A}}

\newcommand{\bfo}{{\bf 1}}
\newcommand{\rank}{\operatorname{rank}}

\newcommand{\cI}{{\cal I}}
\newcommand{\cU}{{\cal U}}
\newcommand{\cV}{{\cal V}}

\newcommand{\as}{{\em a.s.}}

\newcommand{\rint}{\operatorname{int}}
\newcommand{\supp}{\operatorname{supp}}

\newcommand{\xc}[1]{\vspace{.3cm}

\noindent {\em #1} }
\newcommand{\xcb}[1]{\vspace{.0cm}

\noindent {\em #1} }
\newcommand{\mab}[1]{\vspace{.1cm}

\noindent {\bf #1} }

\usepackage{tikz}
\usetikzlibrary{shapes,arrows,graphs,graphs.standard,shapes.misc}
\usetikzlibrary{matrix,decorations.pathreplacing, calc, positioning,fit}
\usetikzlibrary{shapes.geometric}
\usetikzlibrary{decorations.markings}

\newtheorem{definition}{Definition}

\makeatletter
\def\blfootnote{\xdef\@thefnmark{}\@footnotetext}
\makeatother

\def\BibTeX{{\rm B\kern-.05em{\sc i\kern-.025em b}\kern-.08em
    T\kern-.1667em\lower.7ex\hbox{E}\kern-.125emX}}

\begin{document}

\title{Geometric Characterization of the $H$-property for Step-graphons}
\author{Mohamed-Ali Belabbas\footnote{M.-A.~Belabbas is with the Department of Electrical and Computer Engineering and the Coordinated Science Laboratory, University of Illinois, Urbana-Champaign. Email:  \texttt{belabbas@illinois.edu}.}\quad  and\quad Xudong Chen\footnote{X.~Chen is with the Department of Electrical, Computer, and Energy Engineering, University of Colorado Boulder. Email: \texttt{xudong.chen@colorado.edu}.}
}

\date{}

\maketitle

\begin{abstract}
\blfootnote{M.-A. Belabbas and X. Chen contributed equally to the manuscript in all categories.}
    In a recent paper~\cite{bcb2021h}, we have exhibited a set of conditions that are necessary for the $H$-property to hold for the class of step-graphons. In this paper, we prove that these conditions are essentially sufficient.\end{abstract}

\section{Introduction and Main Result}\label{sec:intro}
In~\cite{bcb2021h}, we introduced the so-called $H$-property for a graphon $W$ --- roughly speaking, it is the property that a graph $G$ sampled from $W$ admits a Hamiltonian decomposition \as.
The presence of a Hamiltonian decomposition in a graph underlies a number of important properties pertinent to structural system theory, such as structural controllability~\cite{chen2021sparse} and structural  stability~\cite{belabbas_algorithmsparse_2013,belabbas2013sparse}. See~\cite{bcb2021h} for details.  In that same paper, we have exhibited a set of conditions that were necessary for the $H$-property to hold for the class of step-graphons. We show in this paper that these conditions are also essentially (in a sense made precise below) sufficient.
\vspace{.2cm}

\noindent
{\bf $H$-property:} 
We start by recalling the definitions of a graphon and its sampling procedure.  
A {\em graphon} is a symmetric, measurable function $W: [0,1]^2\to [0,1]$. Step-graphons, along with their partitions, are defined below:

\begin{definition}[Step-graphon and its partition]\label{def:stepgraphon}
A graphon $W$ is a {\bf step-graphon} if there exists an increasing sequence $0 = \sigma_0 < \sigma_1< \cdots < \sigma_q = 1$ such that $W$ is constant over each rectangle $[\sigma_{i}, \sigma_{i + 1})\times [\sigma_{j}, \sigma_{j + 1})$ for all $0\leq i, j\leq q-1$.  We call $ \sigma = (\sigma_0,\sigma_1,\ldots,\sigma_q)$ a {\bf partition} for $W$.
\end{definition}

Graphons can be used to sample undirected graphs. Other uses of graphons in system theory as limits of  adjacency matrices can be found in~\cite{gao2019graphon, gao2021linear,parise2021analysis}. In this paper, we denote by  $G_n \sim W$ graphs $G_n$ on $n$ nodes  sampled from a graphon $W$. The sampling procedure was introduced in~\cite{lovasz2006limits, borgs2008convergent} and is reproduced below: Let $\U[0,1]$ be the uniform distribution on $[0,1]$. Given a graphon $W$, a graph $G_n=(V,E)\sim W$ on  $n$ nodes is obtained as follows: 
\begin{enumerate}
    \item Sample $y_1,\ldots,y_n\sim \U[0,1]$ independently. 
    We call $y_i$ the {\em coordinate of node} $v_i\in V$. 

    \item For any two distinct nodes $v_i$ and $v_j$, place an edge $(v_i,v_j) \in E$ with probability $W(y_i,y_j)$.
\end{enumerate}
It should be clear that if $0\leq p\leq 1$ is a {\em constant} and $W(s,t)=p$ for all $(s,t)\in [0,1]^2$, then $G_n \sim W$ is  an Erd\H os-R\'enyi random graph with parameter~$p$. Consequently,  graphons can be seen as a way to introduce {\em inhomogeneity} in the edge densities between different pairs of nodes.

Let $W$ be a graphon and $G_n \sim W$. 
In the sequel, we use the notation $\vec G_n=(V,\vec E)$ to denote the {\em directed} version of $G_n$, defined by the edge set
\begin{equation}\label{eq:defvecE}
\vec E :=\{v_iv_j, v_jv_i \mid (v_i,v_j) \in E \}.
\end{equation}
In words, we replace an undirected edge $(v_i,v_j)$ with two directed edges $v_iv_j$ and $v_jv_i$. 
The directed graph $\vec G_n$ is said to have a {\em Hamiltonian decomposition} if it contains a subgraph $H$, with the same node set of $\vec G$, such that $H$ is a node disjoint union of directed cycles. See Fig.~\ref{fig:hamildecom} for illustration. 

\begin{figure}
    \centering
    \subfloat[\label{sfig1:Ggraph}]{
   \includegraphics{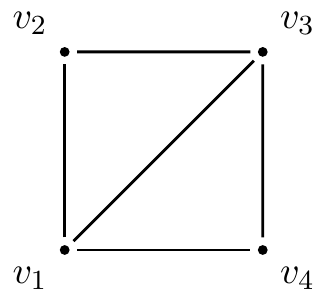}
}
\qquad
\subfloat[\label{sfig1:vecG}]{
  \includegraphics{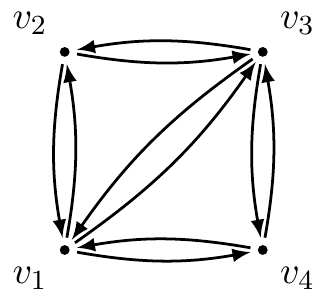}
}
\caption{{\em Left:} An undirected graph $G$ on $4$ nodes. {\em Right:} The directed graph $\vec G$ obtained from $G$ by replacing every undirected with two oppositely oriented edges. The cycle $D_1=v_1v_2v_3v_4v_1$ forms a Hamiltonian decomposition of $\vec G$. The two node disjoint cycles $D_2 = v_1v_2v_1$ and $D_3=v_3v_4v_2$ also form a Hamiltonian decomposition of~$\vec G$. }
    \label{fig:hamildecom}
\end{figure}

We now have the following definition:

\begin{definition}[$H$-property]\label{def:Hproperty}
Let $W$ be a graphon and $G_n\sim W$. 
Then, $W$ has the {\bf $H$-property} if 
\begin{equation}\label{eq:hproperty}
\lim_{n\to\infty}\mathbb{P}(\vec G_n \mbox{ has a Hamiltonian decomposition}) = 1.
\end{equation}
\end{definition}
We will see below that the $H$-property is essentially a ``{\em zero-one}'' property in a sense that the probability on the left hand side of~\eqref{eq:hproperty} converges to either $0$ or $1$. 

\vspace{.2cm}

\noindent 
{\bf Key objects:} We present three key objects associated with a step-graphon, namely, its concentration vector, its skeleton graph, and its associated edge polytope, all of which were introduced in~\cite{bcb2021h}.   

\begin{definition}[Concentration vector]
    Let $W$ be  a {step-graphon} with partition $\sigma = (\sigma_0,\ldots,\sigma_q)$. 
    The associated {\bf concentration vector} $x^* = (x^*_1,\ldots, x^*_q)$ has entries  defined as follows: 
    $x^*_i := \sigma_i - \sigma_{i-1}$, for all $i = 1,\ldots, q$. 
\end{definition}
It should be clear from the sampling procedure above that the concentration vector describes the proportion of sampled nodes in each interval $[\sigma_i,\sigma_{i+1})$. 

Given a step-graphon, its support can be described by a graph, which we call {\em skeleton graph}: 

\begin{definition}[Skeleton graph]\label{def:skeleton}
To a step-graphon $W$ with a partition $\sigma = (\sigma_0,\ldots, \sigma_q)$, we assign the  undirected graph $S = (U, F)$ on $q$ nodes, with $U =\{u_1,\ldots, u_q\}$ and edge set $F$ defined as follows: there is an edge between $u_i$ and $u_j$ if and only if $W$ is non-zero over $[\sigma_{i-1},\sigma_i)\times [\sigma_{j - 1}, \sigma_j)$. 
We call $S$ the {\bf skeleton graph} of $W$ for~$\sigma$.
\end{definition}

We illustrate the relationship between a step-graphon and its skeleton graph in Figure~\ref{fig:stepgraphon}.

\begin{figure}[t]
    \centering
    \subfloat[\label{sfig1:stepgraphon}]{
\includegraphics{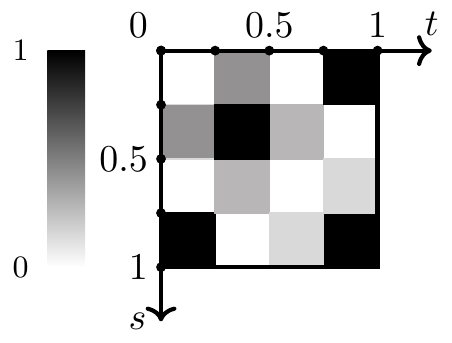}}
\quad
\subfloat[\label{sfig1:skeleton}]{
   \includegraphics{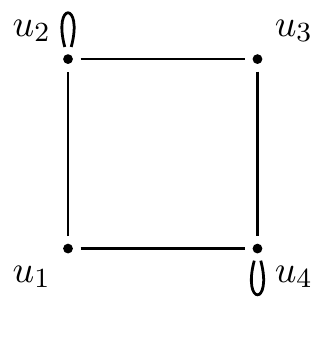}
}
\caption{ {\em Left:} A step-graphon $W$ with the partition $\sigma=(0,0.25,0.5,0.75,1)$, with the value coded by the grayscale on the left. {\em Right:} The associated skeleton graph $S$.}
    \label{fig:stepgraphon}
\end{figure}

Without loss of generality and for ease of presentation, we will consider throughout this paper  step-graphons $W$ whose skeleton graphs are {\em connected}. Though there is no unique skeleton graph associated to a step-graphon (since there are infinitely many different partitions for $W$), we show in Proposition~\ref{prop:invarianceofskeleton} that if one such skeleton graph is connected, then so are all the others. For $S$ not connected, it is not too hard to see that the corresponding step-graphon is block-diagonal. Our results apply naturally to every connected component of~$S$.

We decompose the edge set of $S$ as $F=F_0 \cup F_1$, where elements of $F_0$ are self-loops, and elements of $F_1$ are edges between distinct nodes. We also introduce the subset $F_2 \subseteq F_1$ of edges that are {\em not} incident to two nodes with self-loops.

Let $\mathcal{I} := \{1,\ldots,|F|\}$ be an index set for~$F$ (so that the edges are now ordered). We decompose $\mathcal{I}$ similarly: let $\mathcal{I}_0$, $\mathcal{I}_1$, and~$\mathcal{I}_2$ index $F_0$, $F_1$, and $F_2$ respectively.

To introduce the edge-polytope of $S$, we recall that the incidence matrix $Z =[z_{ij}]$ of $S$ is an $|U| \times |F|$ matrix with its entries defined as follows: 
\begin{equation}\label{eq:defZS}
z_{ij} := \frac{1}{2}
\begin{cases}
    2, & \text{if } f_j\in F_0 \text{ is a loop on node } u_i,       \\
    1, & \text{if node } u_i \text{ is incident to } f_j\in F_1, \\
    0, & \text{otherwise}.
\end{cases}
\end{equation}
Owing to the factor $\frac{1}{2}$ in~\eqref{eq:defZS}, all columns of $Z$ are probability vectors, i.e., all entries are nonnegative and sum to one. The edge polytope of $S$ was introduced in~\cite{ohsugi1998normal} and the definition is reproduced below (with a slight difference in inclusion of the factor $\frac{1}{2}$ of the generators $z_j$):
\vspace{.1cm}

\begin{definition}[Edge polytope]\label{def:edgepolytope}
Let $S = (U,F)$ be a skeleton graph and $Z$ be the associated incidence matrix. Let $z_j$, for $1\leq j \leq |F|$, be the columns of $Z$.  
The {\bf edge polytope} of $S$, denoted by $\cX(S)$, is the convex hull of the vectors $z_j$: 
\begin{equation}\label{eq:defXS}
\cX(S):= \conv\{z_j \mid j = 1,\ldots, |F|\}.
\end{equation}
\end{definition}
A point $x\in \cX(S)$ is said to be in the {\em relative interior} of $\cX(S)$, denoted by $\rint \cX(S)$, if there exists an open neighborhood $U$ of $x$ in $\R^q$ (with $q = |U|$) such that $U\cap \cX(S) \subseteq \cX(S)$. If $x$ is not an interior point, then it is called a {\em boundary} point and we write $x\in \partial \cX(S)$.  

\vspace{.2cm}

\noindent
{\bf Main result:} Let $W$ be a step-graphon. For a given partition $\sigma$ for $W$, let $x^*$ and $S$ be the associated concentration vector and the skeleton graph (which is assumed to be connected). We say that a cycle in $S$ is {\em odd} if it contains an odd number of distinct nodes (or edges); with this definition, self-loops are odd cycles.  Given these, we state the following two conditions:

\vspace{.2cm}

\noindent {\em Condition A:} The graph $S$ has an odd cycle.

\vspace{.2cm}

\noindent {\em Condition B:} The vector $x^*$ belongs to $\rint \cX(S)$.
\vspace{.2cm}

The two conditions are stated in terms of a partition $\sigma$ and its induced skeleton graph and edge-polytope. As mentioned earlier,  there exist infinitely many  partitions for a given step-graphon. However, the following proposition states that the two above conditions  are  {\em invariant} under changes of a partition.

\begin{proposition}\label{prop:invarianceofskeleton}
Let $W$ be a step-graphon. For any two partitions $\sigma$ and $\sigma'$ for $W$, let $x^*$, $x'^*$ be the corresponding concentration vectors and let $S$, $S'$ be the corresponding skeleton graphs. Then, the following hold:
\begin{enumerate}
    \item $S$ is connected if and only if $S'$ is;
    \item $S$ has an odd cycle if and only if $S'$ does;
    \item $x^*\in \cX(S)$ {\em (}resp. $x^* \in\rint \cX(S)${\em )} if and only if $x'^*\in \cX(S')$ {\em (}resp. $x'^*\in \cX(S')${\em )}.
\end{enumerate}
\end{proposition} 

We refer the reader to Appendix~\ref{app:prop1} for a proof of the proposition. 

We are now in a position to state the main result:

\begin{main}\label{thm:sufficient}
Let $W$ be a step-graphon. If it satisfies Conditions~A and~B for a given (and, hence, any) partition~$\sigma$, then it has the $H$-property. 
\end{main}

\begin{remark}\normalfont 
In our earlier work~\cite{bcb2021h}, we have shown that if a step-graphon $W$ has the $H$-property, then it is necessary that Condition A and the following hold: 
\vspace{.1cm}

\noindent {\em Condition B':} The vector $x^*$ belongs to $\cX(S)$.
\vspace{.1cm}

\noindent
Note that condition~B' is weaker than Condition B:  Specifically, Condition~B leaves out the set of step-graphons for which $x^* \in \partial \cX(S)$, which is a set of measure zero. For step-graphons satisfying Conditions~A and~B', but not~B, it is possible that
 $$\lim_{n \to \infty}\mathbb{P}(\vec G_n \sim W \mbox{ has a Hamiltonian decomposition}) \in (0,1).$$
 We have produced explicit examples of such step-graphon  in~\cite{BCBASCC22,bcb2021h}.\hfill \qed
\end{remark}

\mab{Outline of proof:} 
Given a step-graphon $W$ with skeleton graph $S$, and $G_n \sim W$, the sampling procedure induces a natural graph homomorphism $\pi: G_n\to S$, whereby all nodes $v_j$ of $G_n$ whose coordinates $y_j$ belong to $[\sigma_{i-1},\sigma_i)$ are mapped to $u_i$. 
With a slight abuse of notation, we will use the same letter $\pi$ to denote the homomorphism $\pi:\vec G_n \to \vec S$.  

Let $n_i(G_n):= |\pi^{-1}(u_i)|$ be the number of nodes whose coordinates belong to $[\sigma_{i-1},\sigma_i)$. We call the following vector the {\em empirical concentration vector} of $G_n$:
    \begin{equation}\label{eq:defecv}x(G_n):= \frac{1}{n}(n_1(G_n),\ldots, n_q(G_n)).
\end{equation}
The proof of the Main Theorem contains three steps,  outlined below, among which step 2 contains the bulk of the proof. 

\xc{Step 1:}
The proof starts by showing how conditions $A$ and $B$ imply that the empirical concentration vector eventually belongs to the edge polytope. 
First, it should be clear that the edge polytope $\cX(S)$ is a subset of the standard simplex $\Delta^{q-1}$ in $\R^q$; thus, $\dim \cX(S) \leq (q-1)$. Condition A, owing to~\cite{ohsugi1998normal}, is both necessary and sufficient for the equality to hold. Next, note that $nx(G_n) = (n_1(G_n),\ldots, n_q(G_n))$ is a multinomial random variable with $n$ trials and $q$ outcomes with probabilities $x^*_i$, for $1 \leq i \leq q$. 
Then, Condition B guarantees, via  Chebyshev's inequality, that $x(G_n)$ belongs to $\rint \cX(S)$ \as~--- in this paper, \as~stands for almost surely as $n \to \infty$. 

The next two steps are then dedicated to establishing the following fact:
    \begin{equation}\label{eq:nicestatement}
    x(G_n) \in \rint \cX(S)   \\ \Rightarrow  G_n\mbox{ admits a  Hamiltonian decomposition \as.}
    \end{equation}

\xc{Step 2:}
    We start by working under the assumption  that $W$ is a {\em binary} step-graphon, i.e., $W(s,t) \in \{0,1\}$ for almost all $(s,t)\in [0,1]^2$.  In this case, we will see that a sampled graph $G_n \sim W$ is {\em completely determined} by its empirical concentration vector $x(G_n)$. 
    Consequently, our task (establishing~\eqref{eq:nicestatement}) is reduced to establishing the following: 
    \begin{multline}\label{eq:nicestatementbsg}
    x(G_n) \in \rint \cX(S) \mbox{ and $n$ is sufficiently large} \\ \Rightarrow G_n\mbox{ admits a  Hamiltonian decomposition surely.}
    \end{multline}  
       The proof of~\eqref{eq:nicestatementbsg} is constructive. 
       An important object that will arise therefrom  is what we call the $A$-matrix assigned to every Hamiltonian decomposition $H$ for $\vec G_n$, written as a map $\rho(H)=A$.    
   
    Specifically, the matrix $A$ is a $q$-by-$q$ matrix whose $ij$th entry tallies the number of edges of $H$  that go from a node in $\pi^{-1}(u_i)$ to a node in $\pi^{-1}(u_j)$ (A precise definition is in Subsection~\ref{ssec:edgepoly} and see Figure~\ref{fig:amatrix} for an illustration). Any such matrix is then shown to satisfy a number of enviable properties, among which we have $\rho(H)\bfo = x(G_n)$. 
 
    In a nutshell, we have just created the following sequence of maps: 
    $$
    H \xmapsto{} \rho(H)\xmapsto{}  \rho(H)\bfo = x(G_n),
    $$
    with the {\em domain} being all Hamiltonian decompositions in $\vec G_n$, for any $G_n$ sampled from a given binary graphon.

    Now, the effort in establishing~\eqref{eq:nicestatementbsg} is to create appropriate right-inverses (at least locally) of the maps in the above sequence, i.e., we aim to create maps $x \mapsto A(x)$ and $\tilde \rho: A(x) \mapsto H$ with the property that $\rho\cdot \tilde \rho$ is the  identity map and $A(x)\bfo =x$. The map $x\mapsto A(x)$ is created in Proposition~\ref{prop:propnijn}, Subsection~\ref{ssec:xtoA}, and the map $\tilde \rho$ is created in Proposition~\ref{prop:Aham},  Subsection~\ref{ssec:constructionHfromA}. 
    From these two subsections, it will be clear that  by introducing the $A$-matrix as an intermediate object, we can decouple the analytic part of the proof, contained in the creation of the map $x\mapsto A(x)$, from the graph-theoretic part, contained in the creation of $\tilde \rho$. This will conclude the proof of~\eqref{eq:nicestatementbsg}.

\xc{Step 3:} To close gap between {\em binary} step-graphons and general ones, we introduce here an equivalence relation on the class of step-graphons, namely, two step-graphons $W_1$ and $W_2$ are equivalent if their supports are the same. Or, said otherwise, $W_1$ and $W_2$ are equivalent if they share the same concentration vector and skeleton graph. 
Note, in particular, that each equivalence class $[W]$ contains  a unique representative which is a {\em binary} step-graphon, denoted by  $W^s$. We then establish~\eqref{eq:nicestatement} by showing that $W$ has the $H$-property if and only if $W^s$ does. In essence, we show that  the $H$-property is decided completely by the concentration vector and the skeleton graph of a step-graphon $W$.  

The proof of this statement builds upon several classical results from random graph theory, and is presented in Subsection~\ref{ssec:proofth2}.

\vspace{.2cm}

\noindent
{\bf Notation:} 
We gather here key notations and conventions.  

\vspace{.1cm}

\noindent {\em Graph theory:} 
Let $G=(V,E)$ be an undirected  graph. Graphs in this paper do {\em not} have multiple edges, but may have self-loops. We denote edges by $(v_i, v_j)\in E$; if $v_i = v_j$, then we call the edge a self-loop.   
For a given node $v_i$, let $N(v_i):=\{v_j \in V \mid (v_i,v_j)\in E\}$ be the {\em neighbor set} of $v_i$. The degree of $v_i$, denoted by $\deg (v_i)$, is the cardinality of $N(v_i)$.

We will also deal with directed graphs (digraphs) in this paper. 
Whether a graph is directed or undirected will be clear from the context and/or notation. 
We denote by $v_iv_j$ the directed edge from $v_i$ to $v_j$; we call $v_j$ an {\em out-neighbor} of $v_i$ and $v_i$ an {\em in-neighbor} of $v_j$. 
Similarly, we define $N_+(v_i)$ and $N_-(v_i)$ the sets of {\em in-neighbors} and {\em out-neighbors} of $v_i$, respectively.

Recall that for a given undirected graph $G = (V, E)$, possibly with self-loops, we let $\vec G=(V,\vec E)$ be the directed graph defined  as in~\eqref{eq:defvecE}. Self-loops in $\vec G$ are the same as the ones in $G$, i.e., they are not duplicated.

A {\em closed walk} in a graph (or digraph) is an ordered sequence of nodes $v_1v_2\cdots v_k v_1$ in $G$ (resp. $\vec G$) so that all consecutive nodes are ends of some edges (resp. directed edges). 
A {\em cycle} is a closed walk without repetition of nodes in the sequence except the starting- and the ending-nodes. For clarity of the presentation, we use letter $C$ to denote cycles in undirected graphs and the letter $D$ for cycles in directed graphs. 

\xc{Miscellaneous:} We use $\mathbf{1}$ to denote a column vector of all $1$, whose dimension will be clear within the context. We write $x\leq y$ for vectors $x,y\in \R^q$ if the inequality holds entrywise. For a given vector $x \in \R^q$, we denote its $\ell_1$ normalization by $\bar x$, i.e., $\bar x := \frac{x}{\| x \|_1}$,  with the convention that $\bar 0 = 0$. 
Further, given the vector $x$, we denote by $[x]$ the vector whose entry $[x]_i$ is a closest integer to $x_i$ for $1\leq i \leq q$ where for the case $x_i=k+\frac{1}{2}$, with $k$ an integer, we set $[x]_i:= k$. 
We denote the standard simplex in $\R^q$ by $\Delta^{q-1}:=\{x\in \R^q\mid x\ge 0 \mbox{ and } x^\top \bfo = 1 \}$. 
Finally, given a $q \times q$ matrix $A$, we denote by $\supp A$ its {\em support}, i.e., the set of indices corresponding to its non-zero entries.

\section{Analysis and Proof of the Main Theorem}
Throughout the proof, $W$ is a step-graphon, $\sigma$ its associated partition,  and $S = (U,F)$ its skeleton graph on $q$ nodes, which has an odd cycle. Let $F_0$ and $F_1$ be given as after Definition~\ref{def:skeleton}. We can naturally associate to them the subgraphs
\begin{equation}\label{eq:defS12}
S_0:= (U,F_0) \quad  \mbox{and} \quad S_1:=(U,F_1). 
\end{equation}

Note that $S$ has an odd cycle if and only if $S_0$ is edgewise non-empty or $S_1$ has an odd cycle. The lemma below states that we can consider,  without loss of generality,  only the latter case of $S_1$ containing an odd cycle.

\begin{lemma}\label{lem:SoddS1odd}
Let $W$ be a step-graphon. If $W$ admits a partition $\sigma$ with skeleton graph $S$ containing an odd cycle, then $W$ admits a partition $\sigma'$ with skeleton graph $S'$ so that the subgraph $S'_1$ has an odd cycle.  
\end{lemma}

The proof of the lemma can  be established by using the notion of ``one-step refinement'' introduced  in Appendix~\ref{app:prop1} for the partition $\sigma$: If $S_1$ already has an odd cycle, then there is nothing to prove. Otherwise, consider a one-step refinement on a node with self-loop in $S$, which will yield a cycle of length~$3$ in $S'$.

\subsection{On the edge polytope $\cX(S)$}\label{ssec:edgepoly}

\mab{Rank of $\cX(S)$:}
Recall that $\cX(S)$ is the edge-polytope of $S$. Similarly, we let $\cX(S_i)$, for $i = 1,2$, be the edge polytope (see Definition~\ref{def:edgepolytope}) of $S_i$, i.e., $\cX(S_i)$ is the convex hull of the $z_j$'s, with $j\in \cI_i$, where $\cI_i$ indexes edges in $F_i$.


We call $x$ an {\em extremal point of a polytope $\cX$} if there is no line segment in $\cX$ that contains $x$ in its interior. The maximal set of extremal points is called the set of {\em extremal generators for $\cX$}. 
The following result characterizes the extremal generators of $\cX(S_0)$, $\cX(S_1)$, and of $\cX(S)$: 

\begin{lemma}\label{lem:extremalgeenrator}
The set of extremal generators of $\cX(S_i)$, for $i = 1,2$, is $\{z_j \mid j\in \mathcal{I}_i\}$. 
The set of extremal generators of $\cX(S)$ is $\{z_i \mid i\in \mathcal{I}_0\cup \mathcal{I}_2\}$. . 

\end{lemma}

\begin{proof}
The statement for $\cX(S_0)$ is obvious from the definition of the corresponding~$z_i$. For $\cX(S_1)$, it suffices to see that the vectors $z_i$, for $i \in \cI_1$,  have exactly  two non-negative entries, and the supports of these vectors are pairwise distinct. Hence, if $z_i = \sum_{j \in \cI_1} c_i z_j$ with $c_j \geq 0$, we necessarily have $c_j=0$ for $j \neq i$ and $c_i=1$. For $\cX(S)$, we refer the reader to~\cite[Proposition~1]{bcb2021h} for a proof.
\end{proof}

The rank of a polytope $\cX$ is the dimension of its relative interior.  
It is known~\cite{ohsugi1998normal} that if $S$ has $q$ nodes, then
\begin{equation}\label{eq:rankXSevenodd}
    \rank \cX(S) = 
    \begin{cases}
    q -1 & \mbox{if $S$ has an odd cycle}, \\
    q - 2 & \mbox{otherwise.}
    \end{cases}
\end{equation}
Equivalently, we have the following result~\cite{van1976incidence} on the rank of the incidence matrix~$Z_S$ of $S$:
\begin{equation}\label{eq:rankZSevenodd}
    \rank Z_S = 
    \begin{cases}
    q  & \mbox{if $S$ has an odd cycle}, \\
    q - 1 & \mbox{otherwise.}
    \end{cases}
\end{equation}

\mab{The $A$-matrix:}
Let $G_n \sim W$ and suppose that $\vec G_n$ has a Hamiltonian decomposition, denoted by $H$.

Recall that $\pi: \vec G_n\to \vec S$ is the graph homomorphism introduced above~\eqref{eq:defecv}. 
Let $n_{ij}(H)$ be the number of (directed) edges of $H$ from a node in $\pi^{-1}(u_i)$ to a node in $\pi^{-1}(u_j)$.   
It is not hard to see that  (see~\cite[Lemma~1]{bcb2021h} for a proof)  
for all $u_i \in U$, 
\begin{equation}\label{eq:solH}
n_i(G_n)=\sum_{u_j \in N(u_i)} n_{ij}(H) = \sum_{u_j\in N(u_i)} n_{ji}(H).
\end{equation}

We  now assign to the skeleton graph $S$ a convex set that will be instrumental in establishing the main result:

\begin{definition}[$A$-matrices and their set]\label{def:CaS}
Let $S = (U,F)$ be an undirected graph on $q$ nodes. We define  $\cA(S)$ as the set of $q\times q$ nonnegative matrices $A = [a_{ij}]$ satisfying the following two conditions: 
\begin{enumerate} 
\item If $(u_i,u_j) \notin F$, then $a_{ij}=0$;
\item $A\bfo = A^\top \bfo$, and $\bfo^\top A \bfo = 1$. 
  \end{enumerate} 
\end{definition}

Because every defining condition for ${\cal A}(S)$ is affine, the set $\cA(S)$ is a convex set.   

Now, to each Hamiltonian decomposition $H$ of $\vec G_n$,  we assign the following $q\times q$ matrix:  
\begin{equation}\label{eq:IDSM}
\rho(H):= \frac{1}{n} \left[ n_{ij}(H)\right]_{1\leq i,j\leq q}.
\end{equation}
It follows from~\eqref{eq:solH} that $\rho(H) \in \cA(S)$ and $\rho(H)\bfo = x(G_n)$. Furthermore, we have established in~\cite[Proposition~4]{bcb2021h} the following connection between the set $\cA(S)$ and the edge polytope $\cX(S)$: 
\begin{equation}\label{eq:ds}
\cX(S) = \{ x\in \R^q \mid x=A\bfo \mbox{ for some } A\in \cA(S) \}.
\end{equation}
We refer the reader to Figure~\ref{fig:amatrix} for illustration. 

\begin{figure}
    \centering
        \subfloat[\label{sfig3:W}]{
\includegraphics{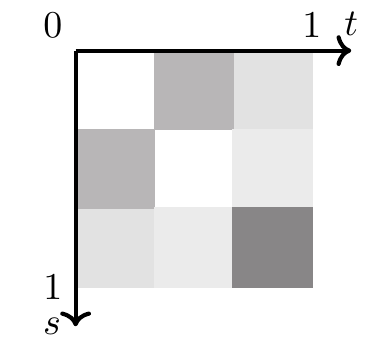}}
\qquad
\subfloat[\label{sfig3:S}]{
   \includegraphics{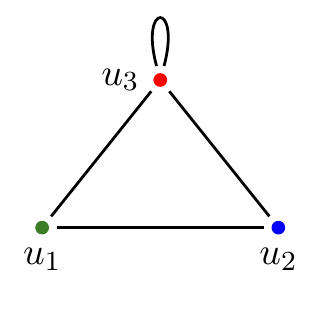}}

\subfloat[\label{sfig3:H}]{
   \includegraphics{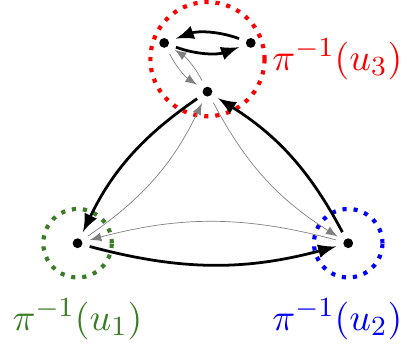}}\quad
\subfloat[\label{sfig2:rho}]{
 \includegraphics{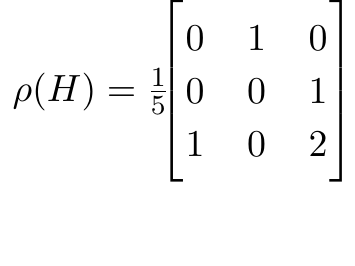}}

\caption{{(a):} A step-graphon $W$. (b): Its skeleton graph $S$. {(c):} A digraph $\vec G_n$, with $G_n\sim W$, for $n = 5$. The edges in black form the Hamiltonian decomposition $H$. (d): The matrix $\rho(H)$ defined in~\eqref{eq:IDSM}.}
    \label{fig:amatrix}
\end{figure}

\subsection{Local coordinate systems on $\cX(S)$ and $\cX(S_1)$}
This section establishes the groundwork for the construction of the map $x\mapsto A(x)$ described in the proof outline. 
To this end, we first show how to express any point in a neighborhood $\cU$ of $x^* \in \cX(S)$ as a {\em positive combination} of the columns of the incidence matrix $Z_S$. This amounts to  solving the linear equations $Z_S \phi(x) = x$, for $x\in \cU$, with $\phi(x)$ being continuous  in $x$ and positive.  We will solve a similar problem for $y^* \in \cX(S_1)$ and with $Z_{S}$ replaced by $Z_{S_1}$, and we call the corresponding solution~$\theta(y)$. These two maps will be put to use in the next subsection.

\mab{Construction of the map $\phi$:} We start with the following lemma: 
\begin{lemma}\label{lem:existphig0}
Suppose that $S = (U,F)$ has an odd cycle. Then, for any $x^*\in \rint \cX(S)$, there exist a closed neighborhood $\mathcal{U}$ of $x^*$ in the simplex and a continuous map $\phi: \mathcal{U} \to \rint \Delta^{|F|-1}$ such that $Z_S \phi(x) = x$ for any $x\in \mathcal{U}$.  
\end{lemma}

\begin{proof}
Because $\cX(S)$ is finitely generated by the columns of $Z_S$, i.e.~the $z_i$'s for $i\in \cI$,  and because $x^*\in \rint \cX(S)$, there exists a {\em positive} probability vector $c := (c_1,\ldots,c_{|F|})$ such that $x^* = Z_Sc$. 
Let $\epsilon:= \frac{1}{2}\min_{i\in \cI} \{c_i\}$. 
Because $S$ has at least one odd cycle, we know from~\eqref{eq:rankZSevenodd} that $Z_S$ is full rank. Thus, we can pick $q$ columns, say $z_1,\ldots, z_q$  of $Z_S$,  that  form a basis of $\R^q$. 
Let $B \subset \R^{|F|}$ be a closed ball centered at $0$ with radius $\epsilon$, and  
let $$B_0:=\{y\in B \mid \bfo^\top y = 0 \mbox{ and } y_i = 0, \mbox{ for all } i > q \}.$$ 
The dimension of $B_0$ is $(q - 1)$.   
We now define the  map $$\psi: B_0 \to \R^q: y \mapsto x^* + Z_Sy = Z_S (y + c).$$ 
It should be clear that $\psi$ is a linear bijection between $B_0$ and its image $\psi(B_0)$. By the construction of $B$ and $B_0$, all the entries of $(y + c)$, for $y\in B_0$, are positive and, moreover,  $\bfo^\top (y + c) = 1$. 
It then follows that  the image of $\psi$ is a closed neighborhood of $x^*$ inside $\cX(S)$. We  now set $\phi := \psi^{-1}$. It remains to show that $\phi(x)\in \Delta^{|F|-1}$. This holds because every column of $Z_S$ belongs to $\Delta^{q-1}$ and so does $x$. Thus, from $Z_S\phi(x) = x$, we have that $x$ is a convex combination of the columns of $Z_S$, which implies that $\phi(x)\in \Delta^{|F|-1}$.
\end{proof}


Let $x^* \in \rint \cX(S)$ and $\phi$ be the map given in Lemma~\ref{lem:existphig0}. 
For an edge  $f_i$ of $S$, we let $\phi_i$ be the corresponding entry of $\phi$. 
We next define two functions $\tau_i: \mathcal{U} \to \R^q$, for $i = 1,2$, as follows: 
\begin{equation}\label{eq:defx0x1}
\tau_i(x): x\mapsto \sum_{j \in \cI_i} \phi_{j}(x) z_j.
\end{equation}
If $S$ has no self-loops, then $\tau_0$ is set to be the zero map. We can thus decompose $x\in \mathcal{U}$ as $$x = \tau_0(x) + \tau_1(x).$$  We record the following simple observation for later use: 

\begin{lemma}\label{lem:supportx0x1}
Let $\cU$ be as in Lemma~\ref{lem:existphig0}. For every $x\in \cU$, the set of indices of nonzero (positive) entries of $\tau_0(x)$ is  $\{i\mid u_i \mbox{ has a self-loop}\}$ and, moreover,  every entry of $\tau_1(x)$ is positive.
\end{lemma}

\begin{proof}
The statement for $\tau_0(x)$ is trivial. The statement for $\tau_1(x)$ follows from the fact that $\phi(x)$ has positive entries and no row of $Z_{S_1}$ is identically $0$. 
\end{proof}

\mab{Construction of the map $\theta$:}
For any $x\in \cU$, let $\bar \tau_i(x)$, for $i = 0,1$, be defined as  follows: 
$$
\bar \tau_i(x):= 
\begin{cases}
\tau_i(x) /\|\tau_i(x)\|_1 & \mbox{if } \tau_i(x) \neq 0,\\
0 & \mbox{otherwise}.
\end{cases}
$$
Since $S$ has odd cycle, recall that we can assume by Lemma~\ref{lem:SoddS1odd} that $S_1$ has an odd cycle. Thus, by~\eqref{eq:rankXSevenodd}, the rank of $\cX(S_1)$ is $(q-1)$. In particular, it implies that the relative interior of $\cX(S_1)$ is open in $\Delta^{q-1}$. 
Further, by Lemma~\ref{lem:supportx0x1}, if $x\in \cU$, then $ \bar \tau_1(x) \in \rint \cX(S_1)$. 

The map $\theta$ we introduce below is akin to the map $\phi$ introduced in Lemma~\ref{lem:existphig0}, but defined on a closed neighborhood of $$\bar x_1^*:=\bar \tau_1(x^*) \in \rint \cX(S_1).$$  

\begin{lemma}\label{lem:existtheta}
Suppose that $S$ (and, hence, $S_1$) has an  odd cycle; then, for the given $\bar x^*_1\in \rint \cX(S_1)$, there exist a closed neighborhood $\cV$ of $\bar x_1^*$ in $\Delta^{q-1}$ and a continuous map $\theta: \cV \to \rint \Delta^{|F_1|-1}$ such that $Z_{S_1} \theta(y) = y$ for any $y\in \cV$.
\end{lemma}
The proof is entirely similar to the one of Lemma~\ref{lem:existphig0}, and is thus omitted. 

Because  $\phi$ and $\theta$ are both positive, continuous maps over closed, bounded domains, there exists an  $\alpha \in (0,1)$ so that 
\begin{equation}\label{eq:defalphalowerbound}
\begin{aligned}
    \phi(x) \ge \alpha\bfo   \mbox{ for all } x \in \cU, \\
    \theta(x)\ge \alpha\bfo  \mbox{ for all }  x \in \cV.
\end{aligned}
\end{equation}

\mab{On the image of $\bar \tau_1$:} 
For a given $x^*\in \rint \cX(S)$, the domains of $\phi$ and $\theta$ are closed neighborhoods $\cU$ and $\cV$ of $x^*$ and $\bar x_1^*$, respectively. Later in the analysis, we will pick an arbitrary $x\in \cU$ and apply $\theta$ to $\bar\tau_1(x)$. 
For this, we need that $\bar \tau_1(x)$ belongs to $\cV$. To this end, we will shrink $\cU$ so that $\bar \tau_1(\cU) \subseteq \cV$ and thus the composition $\theta \bar \tau_1$ is well defined. In fact, we have the stronger statement:

\begin{lemma}\label{lem:xbx1epsilon}
Let $\alpha > 0$ be given as in~\eqref{eq:defalphalowerbound}. 
There exist a closed neighborhood  $\cU'\subseteq \cU$ of $x^*$  and a positive $\epsilon < \frac{1}{4}\alpha$, such that 
$$\overline{\tau_1(x)  +\eta} = \frac{\tau_1(x)  +\eta}{\|\tau_1(x)  +\eta\|_1}\in \cV,$$
for any $x\in \cU'$ and for any $\eta \in \R^q$ with $\|\eta\|_\infty \le \epsilon$.   
\end{lemma}

\begin{proof}
Let $\cV'$ be a closed ball centered at $\bar x_1^*$ and  contained in the interior of $\cV$. Then, it is known~\cite[Theorem 4.6]{munkres2018analysis} that there exists an $\epsilon'> 0$ such that the $\epsilon'$-neighborhood of $\cV'$, with respect to the infinity norm, is  contained in the interior of~$\cV$.
Let $\cU':= \bar\tau_1^{-1}(\cV')$ and $\epsilon$ be sufficiently small so that 
\begin{equation}\label{eq:condepsilon}
\frac{(8+4\alpha)\epsilon}{q\alpha^2} < \min\left \{\epsilon', \frac{1}{4}\alpha \right\}.  
\end{equation}
We claim that the above-defined $\cU'$ and $\epsilon$ are as desired.

Since $\bar \tau_1$ is continuous and since $\cV'$ is a closed ball centered at $\bar x^*_1$, $\cU'$ is a closed neighborhood of $x^*$. Now, pick an arbitrary $x\in \cU'$. 
For ease of notation, we set $x_1 := \tau_1(x)$ and $\bar x_1 = \bar \tau_1(x)$ for the remainder of this proof. 
Then, 
\begin{align}\label{eq:barxdiff}
\left\| \overline {x_1 + \eta} - \bar x_1 \right \|_\infty  &= \left \| \frac{x_1+\eta}{\|x_1+\eta\|_1} - \frac{x_1}{\|x_1\|_1} \right \|_\infty \notag\\
& = \left \| \frac{(\|x_1\|_1 - \|x_1 + \eta\|_1 )x_1 + \|x_1\|_1 \eta}{\|x_1 + \eta\|_1 \|x_1\|_1}  \right \|_\infty  \notag\\
&\leq  \frac{|\|x_1\|_1 - \|x_1 + \eta\|_1 |}{\|x_1 +\eta\|_1 \|x_1\|_1} \|x_1\|_\infty +  \frac{\|\eta\|_\infty}{\|x_1 + \eta\|_1}  \notag\\
&\leq  \frac{|\|x_1\|_1 - \|x_1 + \eta\|_1|}{\|x_1 +\eta\|_1 \|x_1\|_1} + \frac{\|\eta\|_\infty}{\|x_1 + \eta\|_1}
\end{align}
where we used the fact that $\|x_1\|_{\infty} \leq 1$ to obtain the last inequality.
To further evaluate~\eqref{eq:barxdiff}, we first note that
$$
|\|x\|_1 - \|x_1 + \eta\|_1| \le \|\eta\|_1 \le q\|\eta\|_\infty \leq q\epsilon. 
$$
Next, by~\eqref{eq:defZS},~\eqref{eq:defx0x1} and~\eqref{eq:defalphalowerbound}, 
every entry of $x_1$ is greater than $\frac{1}{2}\alpha$, so $\|x_1\|_1 \ge \frac{1}{2} q \alpha$. 
Moreover, since $\epsilon < \frac{1}{4}\alpha$, 
$$
\|x_1 + \eta\|_1 \ge \|x_1\|_1 - \|\eta\|_1 \ge \frac{1}{2} q\alpha - q\epsilon \ge \frac{1}{4} q \alpha.  
$$
Finally, using~\eqref{eq:condepsilon}, we can proceed from~\eqref{eq:barxdiff} and obtain that 
$$
\left\| \overline {x_1 + \eta} - \bar x_1 \right \|_\infty \le \frac{8 \epsilon }{q \alpha^2} + \frac{4\epsilon}{ q\alpha} = \frac{(8 + 4\alpha)\epsilon}{q\alpha^2}< \epsilon',
$$
which implies that $\overline{x_1 + \eta}$ belongs to the $\epsilon'$-neighborhood of $\cV'$ and, hence, to $\cV$. 
\end{proof}

\begin{remark}\label{rem:cucup}
\normalfont From now on, to simplify the notation, we denote by $\cU$ the set $\cU'$ of Lemma~\ref{lem:xbx1epsilon}. \hfill \qed
\end{remark}

\subsection{Construction of the map $x\mapsto A(x)$}\label{ssec:xtoA}
To construct the map, we first specify its domain, which will be a subset of $\cU$. If $x\in \R^q$ is an empirical concentration vector of some $G_n\sim W$ for a step-graphon $W$, then $nx$ necessarily has integer entries. 
Define a subset of $\cU$ as follows: 
\begin{equation}\label{eq:domainofomega1}
    \cU^*:= \left \{x\in \cU \mid nx \in \mathbb{Z}^{q}_+ \mbox{ for some }n\in \mathbb{Z}_+  \right \}, 
\end{equation}
where $\mathbb{Z}_+$ is the set of positive integers. 

Since the analysis for the $H$-property will be carried out in the asymptotic regime $n\to\infty$, the relevant  empirical concentration vectors are those of $G_n$ for $n$ large. 
To this end, let $\alpha\in (0,1)$ be such that~\eqref{eq:defalphalowerbound} is satisfied and $\epsilon$ be as in Lemma~\ref{lem:xbx1epsilon}. We have the following definition:

\begin{definition}\label{def:npairedx}
    Given an $x\in \cU^*$, $n\in \mathbb{Z}_+$ is  {\bf paired with} $x$ if $n > \frac{8}{\alpha\epsilon}$ and $nx$ is integer valued.  
\end{definition}

With the above, we now state the main result of this subsection:

\begin{proposition}\label{prop:propnijn}
Let $S = (U,F)$ be a connected undirected graph, with at least one odd cycle, and $\vec S = (U, \vec F)$ be its directed version.  
Then, there exist a map $A: \cU^* \mapsto \cA(S)$ and a positive number $\underline a$ such that for any $x\in \cU^*$ and for any $n$ paired with $x$, the following hold:
\begin{enumerate}
\item $A(x)\bfo = x$;
\item $nA(x)$ is integer-valued and $n \diag A(x)$ has even entries;
\item $n \|\diag A(x) -  \tau_0(x)\|_\infty \le 1$, where $\tau_0$ is defined in~\eqref{eq:defx0x1};   
\item $n |a_{ij}(x) - a_{ji}(x)|\le 1$ for all $1\leq i, j\leq q$; 
\item\label{it:underlinea}  
For any $u_iu_j \in \vec F$,  $a_{ij}(x) > \underline a$. 
\end{enumerate}
\end{proposition}

Note that item~\ref{it:underlinea} and the fact that $A(x)\in \cA(S)$ imply $\supp A(x)= \vec S$, i.e., $a_{ij} \neq 0$ if and only if $u_iu_j \in \vec F$. 

The proof of Proposition~\ref{prop:propnijn} is constructive. It will rely on a few technical facts we establish here. Let $x \in \cU^*$ and $n$ be paired with $x$.  Define  
\begin{equation}\label{eq:defx'0x'1}
\tau'_0(x):=\frac{2}{n}\left[\frac{n}{2} \tau_{0}(x) \right] \quad \mbox{and} \quad  \tau'_1(x):= x - \tau'_0(x),
\end{equation}   
where we recall that the operation $\left [\, \cdot \, \right]$ is the  integer-rounding operation, introduced in the notation of Section~\ref{sec:intro}. 
The vector $n\tau_0'(x)$ is then the vector with  {\em even}  entries closest to the entries of $n\tau_0(x)$. 
Next, we define
\begin{equation}\label{eq:defn0n1}
n'_0:=n \|\tau'_0(x)\|_1 \mbox{ and } n'_1:=n \|\tau'_1(x)\|_1
\end{equation}
Recall that  $\bar \tau'_0$ and $\bar \tau'_1$ are the $\ell_1$ normalization of $\tau'_0$ and $\tau'_1$, respectively. It should be clear from the construction that $n'_0 + n_1' = n$ and
$$
\bar \tau'_i = \frac{n}{n'_i} \tau'_i, \mbox{ for } i = 1,2.
$$ 
For a given $x\in \cU^*$, there obviously exist infinitely many positive integers $n$ that are paired with $x$. However, the ratios $n'_0/n$ and $n'_1/n$ are independent of $n$ and determined completely by~$x$. 

We also need the following lemma:
\begin{lemma}\label{lem:nijnprelim}
The following items hold:
\begin{enumerate}
    \item For $i = 1,2$, $\supp \tau'_i(x) = \supp \tau_i(x)$ and, moreover, the nonzero entries of $\tau'_i(x)$  are uniformly bounded below by $\frac{1}{2}(1-\epsilon/4)\alpha$.
    \item The ratio $n'_0/n$ is bounded below by $(1 -  \epsilon/8) \alpha |F_0|$. 
    The ratio $n_1'/n$ is bounded below by $3q\alpha/8$. 
    \item Let $\cV$ be defined as in Lemma~\ref{lem:existtheta}. Then, $\bar \tau_1(x)\in \cV$. 
\end{enumerate}
\end{lemma}

We provide a proof of the lemma in Appendix~\ref{app:nijnprelim}. 

With the lemma above, we now establish Proposition~\ref{prop:propnijn}:  

\begin{proof}[Proof of Proposition~\ref{prop:propnijn}]
We start by defining two matrix-valued functions $A_0(x)$ and $A_1(x)$ so that for any $x \in \cU^*$, $A_i(x) \in \cA(S)$ and ${A_i}(x)\bfo=\bar \tau'_i(x)$.  
We will then let $A(x)$ be the convex combination of these two matrices given by 
\begin{equation}\label{eq:defAinprop} A(x) = \frac{n'_0}{n} A_0(x) +\frac{n_1'}{n} A_1(x). 
\end{equation}   
Since $\cA(S)$ is convex, it will then follow that $A(x) \in \cA(S)$.

\xc{Construction of $A_0$:} The matrix $A_0(x)$ is simply given by
\begin{equation}\label{eq:defA0x}
A_0(x):=\diag \bar \tau'_0(x).
\end{equation} 
By Lemma~\ref{lem:supportx0x1}, $\supp \tau_0(x)$ is constant over $\cU$. By the first item of Lemma~\ref{lem:nijnprelim}, $\supp \tau_0(x) =\supp \tau'_0(x)$ for all $x\in \cU^*$. It then follows that $\supp A_0(x)$ is also constant over $\cU^*$. By the same item,  the nonzero entries of $A_0(x)$ are uniformly bounded below by a positive constant.

\xc{Construction of $A_1$:} The construction is more involved than the one of $A_0$, and requires to first define the intermediate matrix $A'_1$.  
To this end, recall that $Z_{S_1}$ is the edge-incidence matrix of $S_1=(U,F_1)$, obtained by removing the self-loops of $S$, and that $\theta$ is the map given in Lemma~\ref{lem:existtheta}, i.e., $Z_{S_1}\theta(y)=y$ for all $y\in \cV$. 
Given an edge $f=(u_i, u_j) \in S_1$, we denote by $\theta_{f}(y)$ the corresponding entry of $\theta(y)$.  
By item~3 of Lemma~\ref{lem:nijnprelim}, $\bar \tau'_1(x)$ belongs to $\cV$, which is the domain of $\theta$. 
Now, we define the symmetric matrix $A'_1(x) = [a'_{1,ij}(x)]\in \R^{q \times q}$ as follows:
\begin{equation}\label{eq:rhotoa}
a'_{1,ij}(x) :=
\begin{cases}
\frac{1}{2}\theta_{f}(\bar \tau'_1(x)) & \mbox{if } f=(u_i,u_j) \in F_1, \\
0 & \mbox{otherwise}.
\end{cases}
\end{equation}
In particular, the diagonal of $A'_1$ is $0$, and so will be the diagonal of $A_1$ as shown below. 
From the definition of the incidence matrix $Z_{S_1}$ and~\eqref{eq:rhotoa}, we have that $A'_1(x) \bfo  = Z_{S_1}\theta(\bar \tau'_1(x))$. By Lemma~\ref{lem:existtheta}, $Z_{S_1}\theta(\bar \tau'_1(x))=\bar \tau'_1(x)$. It then follows that 
\begin{equation}\label{eq:A1xbfox1}
A'_1(x)\bfo = \bar \tau'_1(x),
\end{equation}
and, hence, $\bfo^\top A'_1(x) \bfo = \bfo^\top \bar \tau'_1(x) = 1$.  
Furthermore, since $A'_1(x)$ is symmetric,  ${A'_1(x)}\bfo=A'_1(x)^\top \bfo$. We thus have that 
$A'_1(x)\in \cA(S)$. Since $\theta_f$ is positive for every $f\in F_1$, $\supp A'_1(x)$ is constant over $\cU$. Moreover, by~\eqref{eq:defalphalowerbound}, the nonzero entries of $A'_1(x)$ are uniformly bounded below by $\frac{1}{2}\alpha$. 

Next, we use $A'_1$ to construct $A_1$. There are two cases; one is straightforward and the other is more involved:  
\xc{Case 1:   $n_1'A'_1(x)$ is integer-valued:}  Set $A_1(x):= A'_1(x)$. 

\xc{Case 2: $n_1'A'_1(x)$ is not integer-valued:} In this case, we appeal to the result~\cite[Theorem 2]{belabbas2021integer}: 
There, we have shown that there exists a matrix $A_1(x) = [a_{1,ij}(x)]$ in $\cA(S)$, with 
\begin{equation}\label{eq:eq:A1xbfox1v2}
    A_1(x)\bfo=A'_1(x)\bfo = \bar \tau'_1(x)
\end{equation}
such that $A_1(x)$ has the same support as $A'_1(x)$ and 
$$n_1' a_{1,ij}(x) = \lfloor n_1' a'_{1,ij}(x)\rfloor \quad \mbox{or} \quad n_1' a_{1,ij}(x) = \lceil n_1' a'_{1,ij}(x)\rceil.$$ 
In particular,  $n_1'A_1(x)$ is integer-valued. 
Because $a'_{1,ij}(x) = a'_{1,ji}(x)$, it follows that 
\begin{equation}\label{eq:n1a1le1}
n_1'|a_{1,ij}(x) - a_{1,ji}(x)| \le 1, \quad \forall 1\leq i, j\leq q.   
\end{equation} 
Moreover, if $a'_{1,ij}(x) > 0$, then 
\begin{equation}\label{eq:lowerboundaij}
a_{1,ij}(x) > a'_{1,ij}(x) - \frac{1}{n_1'} \ge \alpha - \frac{1}{n_1'} \ge \alpha - \frac{1}{n}\frac{n}{n_1'}
 \ge \alpha  - \frac{8}{3nq\alpha} > \left (1  - \frac{1}{12 q} \right ) \alpha. 
\end{equation} 
where second to the last inequality follow from item~2 of Lemma~\ref{lem:nijnprelim} and the last inequality follows from the hypothesis on $n$ (specifically $n  > \frac{8}{\alpha\epsilon}$) from the statement and the condition that $\epsilon < \alpha/4$ from Lemma~\ref{lem:xbx1epsilon}.

\xc{Proof that $A$ satisfies the five items of the statement:} 
\begin{enumerate}
    \item From~\eqref{eq:defA0x}, $A_0(x)\bfo = \bar \tau'_0(x)$. For $A_1$, it was shown that $A_1(x) \bfo = \bar \tau'_1(x)$ in~\eqref{eq:A1xbfox1} and~\eqref{eq:eq:A1xbfox1v2} for Case~1 and Case~2, respectively. Since $A$ is the convex combination of $A_0$ and $A_1$ given in~\eqref{eq:defAinprop}, it follows that 
    \begin{equation}\label{eq:Axbfo=x}
    {A(x)\bfo}=\frac{n'_0}{n}\bar \tau'_0(x)+\frac{n_1'}{n}\bar \tau'_1(x) = x.
    \end{equation}

    \item By the construction of $A$ in~\eqref{eq:defAinprop} and the definitions of $A_0$ and $A_1$, the diagonal of $nA(x)$ is $$n'_0A_0(x)=n'_0\Diag \bar \tau'_0(x) = n \Diag \tau'_0(x).$$ By~\eqref{eq:defx'0x'1}, all the entries of $n \tau'_0(x)$ are even. 
    
    \item Using~\eqref{eq:defx'0x'1} again, we have that 
$$-\bfo \le n (\tau_0(x) - \tau_0'(x))  \le \bfo,$$
from which it follows that $$n \|\diag A(x) - \tau_0(x)\|_\infty = n\|\tau'_0(x) - \tau_0(x)\| \le 1.$$ 
    
    \item The off-diagonal entries $a_{ij}(x)$ of $A(x)$ are those of $\frac{n_1'}{n}A_1(x)$, which we denoted by $\frac{n_1'}{n}a_{1,ij}(x)$. 
    Thus, $$n |a_{ij}(x) - a_{ji}(x)| = n_1' |a_{1,ij}(x) - a_{1,ji}(x)| \le 1,$$ where the last inequality is~\eqref{eq:n1a1le1}.
    
    \item {\em Case 1: $S$ does not have a self-loop:} In this case, $A(x) = A_1(x)$. By construction of $A_1$, $\supp A_1(x) = \vec S$. If $A_1(x)$ is obtained via case~1 above, then, as argued after~\eqref{eq:A1xbfox1}, its nonzero entries are bounded below by $\alpha/2$. 
    Otherwise, $A_1$ is obtained via case~2 and its nonzero entries are lower bounded as shown in~\eqref{eq:lowerboundaij}. 
    \vspace{.1cm}
    
    {\em Case 2: $S$ has at least one self-loop:} In this case, $$\supp A(x) = \supp A_0(x) \sqcup \supp A_1(x).$$ 
    By construction of $A_0$ and item~1 of Lemma~\ref{lem:nijnprelim}, 
    \begin{equation}\label{eq:suppA0x}
    \supp A_0(x) = \supp \Diag \tau'_0(x) = \supp \Diag \tau_0(x) = \vec S_0,
    \end{equation} 
    where the last equality follows from Lemma~\ref{lem:supportx0x1}. Moreover, the nonzero entries of $A_0(x)$ are bounded below by $\frac{1}{2}(1-\epsilon/4)\alpha$.  
    Also, by construction of $A_1$, 
    \begin{equation}\label{eq:suppA1x}
    \supp A_1(x) = \vec S_1.
    \end{equation}
    Thus, by~\eqref{eq:suppA0x} and~\eqref{eq:suppA1x}, $\supp A(x) = \vec S$. 
    Finally, we verify that the nonzero entries of $A(x)$ are uniformly bounded below by a positive number. By item~2 of Lemma~\ref{lem:nijnprelim}, $n'_0/n$ and $n_1'/n$ are uniformly bounded below by positive numbers (note that $|F_0| \geq 1$ in the current case). Thus, using~\eqref{eq:defAinprop}, the nonzero entries of $A(x)$ are also uniformly lower bounded by a positive number. 
\end{enumerate}
This completes the  proof.
\end{proof}

\subsection{Constructing a Hamiltonian decomposition from $A(x)$}\label{ssec:constructionHfromA}
In this subsection, we construct the map $\tilde \rho: A(x) \mapsto H$ announced in Section~\ref{sec:intro}, where $A(x)$ will be taken from the statement of Proposition~\ref{prop:propnijn} and $H$ is a Hamiltonian decomposition in $G_n\sim W$, with $x$ its empirical concentration vector.   
Throughout this subsection, we assume that $W$ is a {\em binary} step-graphon, i.e., $W$ is valued in $\{0,1\}$. 

Graphs sampled from a binary step-graphon have rather rigid structures as we will describe below. We refer to them as $S$-multipartite graphs, see also Figure~\ref{fig:smulti}:

\begin{definition}[$S$-multipartite graph]\label{def:smultipartite}
    Let $S=(U,F)$ be an undirected graph, possibly with self-loops. An undirected graph $G$ is an {\bf $S$-multipartite graph}  if there exists a graph homomorphism $\pi:G \to S$, so that $$(v_i,v_j) \in E  \Rightarrow (\pi(v_i),\pi(v_j)) \in F.$$ 
    Further, $G$ is a {\bf complete $S$-multipartite graph} if $$(v_i,v_j) \in E  \Leftrightarrow (\pi(v_i),\pi(v_j)) \in F.$$ 
\end{definition}

\begin{figure}
    \centering
     \subfloat[\label{sfig1:S}]{
 \includegraphics{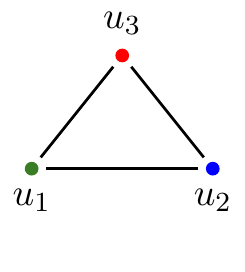}
}\,
\subfloat[\label{sfig1:SMG}]{
 \includegraphics{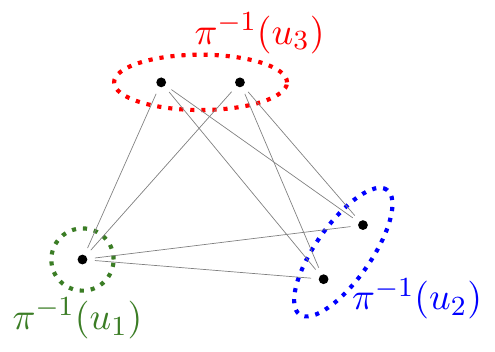}}
    \caption{{\em Left:} An undirected graph $S$ on three nodes. {\em Right:} A complete $S$-multipartite graph $M(\omega,S)$ with $\omega=(1,2,2)$.}
    \label{fig:smulti}
\end{figure}

Let $G$ be an arbitrary complete $S$-multipartite graph with $S=(U,F)$ and set $n_i := |\pi^{-1}(u_i)|$ for $i = 1,\ldots,q$. It should be clear that $G$ is completely determined by $S$ and the vector $w:=(n_1,\ldots,n_q)$ . We will consequently use the notation $M(w, S)$ to refer to a complete $S$-multipartite graph. Now, returning to the case $G_n \sim W$, where $W$ is a binary step-graphon with skeleton graph $S$, the empirical concentration vector $x(G_n)$ together with $S$ then completely determine $G_n$ as announced above.

If $G$ is a (complete) $S$-multipartite graph, then $\vec G$ is (complete) $\vec S$-multipartite, and we use the same notation $\pi$ to denote the homomorphism. We next introduce a special class of cycles in $\vec G$: 

\begin{definition}[Simple cycle]\label{def:simpledecom}
    Let $G$ be an $S$-multipartite graph, and $\pi: \vec G \to \vec S$ be the associated homomorphism. A directed cycle $D$ in $\vec G$ is called {\bf simple} if $\pi(D)$ is a cycle (rather than a closed walk) in~$\vec S$.  
\end{definition}

With the notions above, we state the main result of this subsection:  

\begin{proposition}\label{prop:Aham}
Let $S = (U,F)$ be an undirected graph, possibly with self-loops. 
Let $G = M(nx,S)$ be a complete $S$-multipartite graph on $n$ nodes, where $x\in \cU^*$ and $n$ is paired with $x$ (see Definition~\ref{def:npairedx}). 
Let $A(x)$ be as in Proposition~\ref{prop:propnijn} and  
\begin{equation}\label{eq:defmij}
m_{ij}(x):= n\min\{a_{ij}(x), a_{ji}(x)\},  \mbox{ for all } 1 \leq i,j\leq q.
\end{equation}
Then, there exists a Hamiltonian decomposition $H$ of $\vec G$, with $\rho(H) = A(x)$,  
such that the following hold: 
\begin{enumerate}
    \item\label{it:lowerboundselfloop} There exist exactly $\frac{1}{2} m_{ii}(x)$  disjoint 2-cycles in $H$ pairing $m_{ii}(x)$ nodes in $\pi^{-1}(u_i)$ for every $i = 1,\ldots, q$;
    \item\label{it:lowerboundedges}  There are at least $m_{ij}(x)$ disjoint 2-cycles in $H$ pairing nodes in $\pi^{-1}(u_i)$ to nodes in $\pi^{-1}(u_j)$  for each $(u_i,u_j) \in F_1$.  
    \item\label{it:upboundsimp3} There are at most $\left\lceil\frac{2}{3}|F|\right\rceil$ cycles of length three or more in $H$;  
    \item\label{it:upboundlengthcycle} The length of every cycle of $H$ does not exceed $2|F|$;
    \item\label{it:cycle3simple} All cycles of length at least $3$ of $H$ are simple.
\end{enumerate}
\end{proposition}
We illustrate the Proposition on an example.

\begin{example}\label{ex:propaham}\normalfont
Consider a complete $S$-multipartite graph $G$ for $S$ shown in Figure~\ref{sfig1:S}. Set $n_i := |\pi^{-1}(u_i)|$, for $i= 1,2,3$, $n := \sum_{i = 1}^3 n_i$, and $x:= \frac{1}{n}(n_1,n_2,n_3)$. In this case, $x\in \rint \Delta^{2}$ if and only if the $n_i$'s satisfy triangle inequalities $n_i + n_j > n_k$, where $i$, $j$, and $k$ are pairwise distinct. If these inequalities are satisfied, then $\vec G$ admits a Hamiltonian decomposition $H$, which is comprised primarily (if not entirely) of 2-cycles. We plot in Figure~\ref{fig:example} the corresponding undirected edges of $G$.  Specifically, there are two cases: (1) If $n_1 - n_2 + n_3$ is even, then $H$ is comprised solely of 2-cycles as shown in Figure~\ref{sfig1:FEVEN}. (2) If $n_1 - n_2 + n_3$ is odd, then $H$ is comprised of 2-cycles and a single triangle as shown in Figure~\ref{sfig2:FNOTEVEN}.    
\end{example}

\begin{figure}
    \centering
    \subfloat[\label{sfig1:FEVEN}]{
\includegraphics{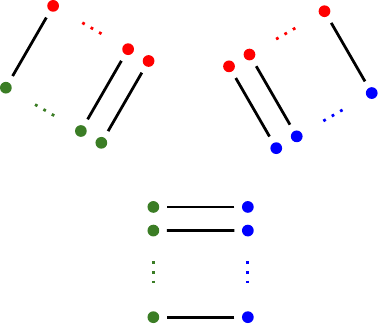}
} \quad 
\subfloat[\label{sfig2:FNOTEVEN}]{
\includegraphics{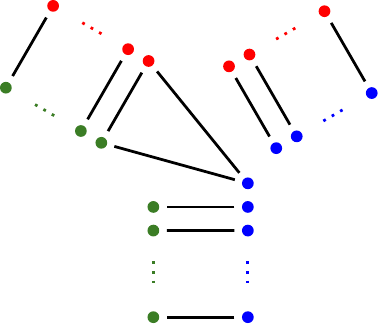}}
    \caption{Illustration of Example~\ref{ex:propaham}: Two complete $S$-multipartite graphs $G$ with $S$ shown in Figure~\ref{sfig1:S}, where the undirected edges plotted in each subfigure give rise to a Hamiltonian decomposition $H$ of $\vec G$. Green nodes belong to $\pi^{-1}(u_1)$, blue nodes to $\pi^{-1}(u_2)$, and red nodes to $\pi^{-1}(u_3)$.  
    }
    \label{fig:example}
\end{figure}

The proof of Proposition~\ref{prop:Aham} relies on a reduction argument for both the graph $G_n$ and the matrix $A(x)$: roughly speaking, we will first remove out of $\vec G_n$ a number of 2-cycles, which leads to a graph $\vec G'$ of smaller size. With regards to the matrix $A$, this reduction leads to another matrix $A'\in \cA(S)$ with the property that $\diag(A') = 0$.  Finding a Hamiltonian decomposition $H$ for $\vec G$ with $\rho(H)=A$ is then reduced to finding a Hamiltonian decomposition $H'$ for $\vec G'$ with $\rho(H') = A'$. 
For the arguments outlined above, we need a supporting lemma stated below, whose proof is relegated to Appendix~\ref{app:prop3}:

\begin{lemma}\label{lem:redH}
Let $n'$ be a nonnegative integer. 
Let $A' \in \cA(S)$ be such that $n'A'$ is integer-valued, $\diag(A')=0$, and  $x':={A'}\bfo$. Then, there exists a Hamiltonian decomposition $H'$ of $\vec G'$, where $G' :=M(n'x',S)$ is the complete $S$-multipartite graph, such that $\rho(H') = A'$ and every cycle in $H'$ is simple. 
\end{lemma}

With the lemma above, we now establish Proposition~\ref{prop:Aham}:

\begin{proof}[Proof of Proposition~\ref{prop:Aham}]
We construct a Hamiltonian decomposition $H$ with the desired properties in two steps. We will fix $x$ in the proof and, to simplify the notation, we omit writing the argument $x$ for $a_{ij}(x)$, $m_{ij}(x)$, and $A(x)$.

\xc{Step 1:} We claim that the following selection of 2-cycles out of $\vec G_n$ is feasible: 
\begin{itemize}
    \item For every self-loop $(u_i,u_i)\in F_0$, $m_{ii} = na_{ii}$ is an even integer, and we select $m_{ii}$ pairwise distinct nodes in $\pi^{-1}(u_i)$ that form $m_{ii}/2$ disjoint 2-cycles.   
    \item For every $(u_i,u_j)\in F_1$, we select $m_{ij}$ distinct nodes in  $\pi^{-1}(u_i)$ and $m_{ij}$ distinct nodes in $\pi^{-1}(u_j)$,  to form $m_{ij}$ disjoint 2-cycles (so the total number of such 2-cycles is $\sum_{(u_i,u_j)\in F_1} m_{ij}$). 
\end{itemize}
The above selection is feasible because (1) $G_n$ is a complete $S$-multipartite graph and, thus, there is an edge between any pair of nodes in $\pi^{-1}(u_i)$, $\pi^{-1}(u_j)$ provided that $(u_i,u_j)\in F_1$ and, (2) $A\bfo = x$ which implies that $\sum^q_{j = 1} m_{ij} \le n_i$ and, hence, we can always pick the required number of distinct nodes.

Let $V'$ be the set of remaining nodes in $G$, i.e., $V'$ is obtained by  removing out of $V$ the 
$\sum_{(u_i,u_j)\in F} m_{ij}$ nodes picked in step 1. 
If $V'$ is the empty set, we let $H$ be the union of the disjoint 2-cycles just exhibited. 
It should be clear that $H$ is a  Hamiltonian decomposition of $\vec G$. 
We claim that $H$ satisfies the desired properties. To see this, let $A':=\rho(H)$ and $x':= {A'}\bfo$. Then, $A'\in \cA(S)$ and, by construction of $H$, $na'_{ij}=m_{ij}$.
On the one hand, since $V'$ is empty, $\|nx'\|_1 = \|nx\|_1 = n$ and, hence, the sum of the entries of $A'$ is equal to the sum of the entries of $A$.
On the other hand, since $m_{ij} \le n a_{ij}$, we have that $A' = \frac{1}{n}[m_{ij}]_{ij} \le A $. It then follows that $A' = A$ and $x' = x$. 
Furthermore, items 1 and 2 follow from Step 1, respectively, and items 3, 4, and 5 hold trivially. 

\xc{Step 2:} 
We now assume that $V'$ is non-empty. Let $G'$ be the subgraph of $G$ induced by $V'$. We exhibit below a Hamiltonian decomposition $H'$ for $\vec G'$ such that the cycles in $H'$, together with the 2-cycles constructed in Step~1, yield a desired~$H$.  
Additionally, we will show that all the cycles of length at least 3 in $H'$ satisfy items 3, 4, and 5.  

To construct the above mentioned $H'$, we will appeal to Lemma~\ref{lem:redH}. 
To this end, let $n'_i$ be the number of nodes in $\pi^{-1}(u_i) \cap V'$, i.e.,
$$n'_i = n_i - \sum_{u_j \in N(u_i)} m_{ij}.$$ 
Let $n':= \sum^q_{i = 1} n'_i$ be the total number of nodes in $G'$, and 
$$x':= \frac{1}{n'}[n'_1;\cdots;n'_q].$$ 
It follows that $G'=M(n'x',S)$.

Because $H$ has to satisfy $\rho(H) = A$ with $A\bfo = x$, $H$ should contain $na_{ij}$ edges from nodes in $\pi^{-1}(u_i)$ to nodes in $\pi^{-1}(u_j)$. 
Since $m_{ij}$ such edges have already been accounted for by the 2-cycles created in Step 1, we need an additional $n'_{ij}$ edges, where 
\begin{equation}\label{eq:defn'ij}
n'_{ij}:= 
\left\{
\begin{array}{ll}
n a_{ij} - m_{ij} & \mbox{if $(u_i,u_j)\in F$}, \\
0 & \mbox{otherwise}.
\end{array}
\right.
\end{equation}
Note that by~\eqref{eq:defmij}, $n'_{ii} = 0$ for all $i = 1,\ldots, q$. 
Correspondingly, we define a $q\times q$ matrix as follows:
\begin{equation*}
A' := \frac{1}{n'} [n'_{ij}].
\end{equation*} 
Because $m_{ij} =  m_{ji}$ for all $(u_i,u_j)\in F$ and because $A\bfo = A^\top \bfo$, we obtain that
$$
\sum^q_{j = 1} n'_{ij} = \sum^q_{j = 1} n'_{ji} = n'_i, \quad\forall i = 1,\ldots, q. 
$$ 
Thus, $A'\in \cA(S)$ and, by construction, $\diag A' = 0$ and ${A'}\bfo = x'$, so $A'$ satisfies the conditions in the statement of Lemma~\ref{lem:redH}.  

By Lemma~\ref{lem:redH}, there exists a Hamiltonian decomposition $H'$ of $\vec G'$ such that $\rho(H') = A'$, ${A'}\bfo=x'$, and all cycles in $H'$ are simple. 
Now, let $H$ be the union of $H'$ and the 2-cycles
obtained in Step 1. Then,   
$$\rho(H) = \frac{1}{n}[m_{ij} + n'_{ij}] = \frac{1}{n}[na_{ij}] = A,$$
where the second equality follows from~\eqref{eq:defn'ij}. 
Moreover, since $\diag A' = 0$, there is no $2$-cycle in $H'$ connecting pairs of nodes in  $\pi^{-1}(u_i)$ for any $i = 1,\ldots, q$. Thus, for each $i = 1,\ldots, q$, $H$ contains {\em exactly} $\frac{1}{2} m_{ii}$ disjoint two-cycles pairing $m_{ii}$ nodes in $\pi^{-1}(u_i)$.

It now remains to show that all the cycles of length at least three in $H'$ satisfy items~3 and~4. To do so, we first provide an upper bound on $n'_i$:   
Using items~3 and~4 of Proposition~\ref{prop:propnijn}, we have that $n a_{ij} - m_{ij}\leq 1$. Thus, 
$$
n'_i  \le n_{i}-\sum_{u_j \in N(u_i)}  (n a_{ij} - 1) =   n_i - n_i + \deg(u_i) =  \deg(u_i),
$$
where $\deg(u_i)$ is the degree of $u_i$ in $S$. 
Since $$\sum^q_{i = 1} \deg(u_i) \leq 2|F|,$$ there are at most $2|F|$ nodes in $\vec G'$. 
Consequently, the length of any cycle in $H'$ is bounded above by $2|F|$ and, moreover, 
there exist at most $\left\lceil\frac{2}{3}|F|\right\rceil$ cycles of length three or more in $H'$. This completes the proof.  
\end{proof}

\begin{remark}\normalfont
The fact that item~2 of the proposition provides a lower bound for the number of 2-cycles instead of an exact number can be understood as follows: The Hamiltonian decomposition $H'$ of $\vec G'$, introduced in Step~2 of the proof, may contain additional 2-cycles pairing nodes from $\pi^{-1}(u_i)$ to $\pi^{-1}(u_j)$, for $(u_i,u_j)\in F_1$.  
\hfill \qed
\end{remark}

\subsection{Proof of the Main  Theorem}\label{ssec:proofth2}

In Subsection~\ref{ssec:constructionHfromA}, we dealt with the construction of a Hamiltonian decomposition in a graph $\vec G_n$  sampled from a {\em binary} step graphon. We will now extend the result to a general step-graphon $W$, for which the existence  of an edge  between a pair of nodes is not a sure event. This will then complete the proof of the Main  Theorem.
 
\begin{figure}
    \centering
 \includegraphics{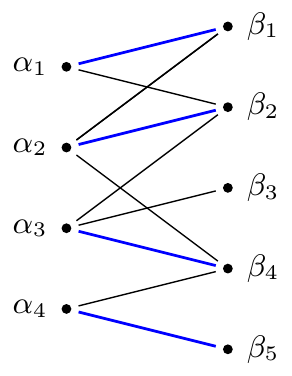}
    \caption{A bipartite graph with $V_L=\{\alpha_1,\ldots,\alpha_4\}$ and $V_R=\{\beta_1,\ldots,\beta_5\}$. The blue edges form a left-perfect matching.}
    \label{fig:bipartite}
\end{figure}

To do so, we first recall some known facts about bipartite graphs. An undirected graph $B=(V,E)$ is called {\em bipartite} if its node set can be written as the disjoint union $V = V_L \sqcup V_R$ so that there does not exist an edge between two nodes in $V_L$ or~$V_R$. Equivalently, a bipartite graph can be viewed as an $S$-multipartite graph where $S$ is a graph with two nodes connected by a single edge.  
We refer to elements of $V_L$ and $V_R$ as {\em left-} and {\em right-nodes}, respectively. A {\em left-perfect matching} $P$ in $B$ is a set of edges so that each left-node is incident to {\em exactly} one edge in $P$, and each right-node is incident to {\em at most} one edge in $P$. 
See Figure~\ref{fig:bipartite} for an illustration. 
Similarly, we define a {\em right-perfect matching} by swapping the roles of left- and right-nodes. 
One can easily see that a left-perfect (resp.~right-perfect) matching exists only if $|V_L| \leq |V_R|$ (resp.~$|V_L| \geq |V_R|$).

Further, we denote by $B(n_1,n_2,p)$ an Erd\H os-R\'enyi {\em random} bipartite graph, with $n_1$ left-nodes, $n_2$ right-nodes, and edge probability $p$ for all edges between left- and right-nodes. 

We need the following fact:   

\begin{lemma}\label{lem:RGfacts}
Let $n_1$ and $n_2$ be positive integers such that $\frac{1}{\kappa}\le  \frac{n_2}{n_1} \le \kappa$, where $\kappa\geq 1$ is a constant. Let $n:= n_1 + n_2$ and $p\in (0,1)$. Then, it holds \as~that the random bipartite graph $B(n_1,n_2,p)$ is connected and contains a left-perfect (resp. right-perfect) matching if $n_2 \geq n_1$ (resp. $n_1\geq n_2$). 
\end{lemma}

The above lemma is certainly well known. For completeness of presentation, we present a proof in Appendix~\ref{app:bipartite}.

We now return to the proof of Main Theorem. For the given step-graphon $W$, we fix a partition $\sigma$, and let $x^*$ be the associated concentration vector and $S$ be the skeleton graph. We now consider a sequence of graphs $G_n \sim W$, with $n\to \infty$.  
We show below that the Hamiltonian decomposition~$H$ for $\vec G_n$ described in Proposition~\ref{prop:Aham} exists \as.

Denote by $W^s$ the {\em saturation} of $W$: it is the binary step-graphon defined as $$W^s(s,t) = 1 \Longleftrightarrow W(s,t) \neq 0.$$ 
We similarly construct a {\em saturated} version of $G_n=(V_n,E_n)\sim W$,  denoted by $G_n^s  = (V_n, E_n^s)$,  as follows: There is an edge $(v_\ell, v_k) \in  E_n^s$ if and only if $(\pi(v_\ell),\pi(v_k))\in F$. Said otherwise, the node set of $G_n$ and $G_n^s$ are the same, but the edges in $G_n^s$ are obtained using the {\em binary} step-graphon $W^s$. It should be clear that $G_n \subseteq G_n^s = M(nx(G_n), S)$, where we recall that $x(G_n)$ is the empirical concentration vector of $G_n$ defined in~\eqref{eq:defecv}. 
  
Let $\cU$ be the closed neighborhood of $x^*$ mentioned in Remark~\ref{rem:cucup}.  
Let $\mathcal{E}_0$ be the event that the empirical concentration vector~$x(G_n)$ of $G_n$ belongs to $\cU$. By Chebyshev's inequality, we have that
$$
\mathbb{P}(\|x(G_n)-x^*\| > \epsilon)\leq \frac{c}{n^2\epsilon^2}, 
$$
which implies that $\mathcal{E}_0$ is almost sure as $n\to \infty$. Thus, we can assume in the sequel that $\mathcal{E}_0$ is true, i.e., the analysis and computation carried out below are conditioned upon $\mathcal{E}_0$.

Note that $nx(G_n)$ is always integer-valued. Since $x(G_n)\in \cU$ by assumption, 
we let $n$ be sufficiently large so that $n$ is paired with $x(G_n)$ (see Definition~\ref{def:npairedx}). We can thus appeal to Proposition~\ref{prop:propnijn} to obtain a matrix $A(x(G_n))$, and to Proposition~\ref{prop:Aham} to obtain a corresponding  Hamiltonian decomposition $H$ of $\vec G_n^s$. 
We now demonstrate that the same  $H$  exists \as~in $\vec G_n$, up to re-labeling of the nodes of $\vec G_n$. The  proof comprises two parts: In part~1, we  show that the cycles in $H$ whose lengths are greater than $2$ exist \as~in $\vec G_n$ and, then, in part~2, we show that the 2-cycles of $H$ do as well.

\vspace{.2cm}
\noindent
{\em Part 1: On cycles of length greater than~2.} For clarity of presentation, we denote by $\pi^s: G_n^s \to S$ the graph homomorphisms associated with $G^s_n$.  
For any {\em path} $u_1\cdots u_k$ in $S$, since $G^s_n$ is complete  $S$-multipartite, there surely exists a {\em path} $v_1\cdots v_k$ in $G^s_n$ so that $\pi^s(v_i)=u_i$. The following result shows that such a  path exist in $G_n$  \as. 

\begin{lemma}\label{lem:pathsimpleGn}
Let $u_1\cdots u_k$ be a path in $S$. Then, it is \as~that there exists a path $v_1\cdots v_k$ in $G_n$, with $\pi(v_i) = u_i$.  
\end{lemma}

\begin{proof}
Since the closed set $\cU$ is in the interior of $\Delta^{q-1}$, 
there exists a $\kappa\geq 1$ such that for all $x\in \cU$, $$\frac{1}{\kappa} \leq \frac{x_i}{x_j}  \leq \kappa, \mbox{ for all } 1\leq i,j \leq q.
$$
Thus, by conditioning on $\mathcal{E}_0$, we have that 
$$
\frac{1}{\kappa} \leq \frac{x_i(G_n)}{x_j(G_n)} = \frac{|\pi^{-1}(u_i)|}{|\pi^{-1}(u_j)|} \leq \kappa, \mbox{ for all }1\leq i, j \leq q.
$$
It then follows that the subgraphs of $G_n$ induced by $\pi^{-1}(u_i) \cup \pi^{-1}(u_{i+1})$ are bipartite and satisfy the hypothesis of Lemma~\ref{lem:RGfacts}, for $1 \leq i \leq k-1$.  
Hence, it is \as~that all of these bipartite graphs are connected.  
We now pick an arbitrary node $v_1 \in \pi^{-1}(u_1)$; 
by the above arguments, we can find $v_2 \in \pi^{-1}(u_2)$ so that $(v_1,v_2)\in G_n$ \as. Iterating this procedure, we obtain the path in $G_n$ sought.
\end{proof}

Now, let $D_1,\ldots, D_m$ be the cycles in $H$ whose lengths are greater than $2$, and $C_1,\ldots, C_m$ be the corresponding undirected cycles in $G_n^s$.  
From items~\ref{it:upboundsimp3} and~\ref{it:upboundlengthcycle} of Proposition~\ref{prop:Aham}, the number $m$ of these cycles, as well as their lengths, are each  uniformly bounded above by constants {\em independent} of $n$.

Let $\ce_1$ be the event that the cycles $C_1,\ldots, C_m$ exist in $G_n$; more precisely,  it is the event that there exist disjoint cycles $C'_i$ in $G_n$ such that $\pi(C'_i) = \pi^s(C_i)$ for all $i = 1,\ldots,m$. 
We have the following lemma:

\begin{lemma}
The event $\ce_1$ is true \as.  
\end{lemma}

\begin{proof}
Let $\ce_{11}$ be the event that there exists a cycle $C'_1 \in G_n$ with $\pi(C'_1)=\pi^s(C_1)$. 
We show that $\ce_{11}$ holds \as.  
To start, we write explicitly $\pi^s(C_1) = u_1\ldots u_k u_1$. 
Since $C_1$ is {simple}, $u_1\ldots u_k u_1$ is a {\em cycle} in $S$.  
By Lemma~\ref{lem:pathsimpleGn}, there exist  \as~nodes $v_i \in \pi^{-1}(u_i)$, for $1\leq i \leq k-1$, such that $v_1\cdots v_{k-1}$ is a path in $G_n$.

In order to obtain the cycle $C'_1$, it remains to exhibit a node  $v_k \in \pi^{-1}(u_k)$ that is connected to both $v_1$ and $v_{k-1}$ in $G_n$. 
We claim that such a node exists with probability  at least
\begin{equation}\label{eq:probclosure}
1 - (1-p_{1k}p_{k-1,k})^{|\pi^{-1}(u_k)|}
\end{equation}
where $p_{ij}>0$ is the value of the step-graphon $W$ over the rectangle $ [\sigma_{i-1},\sigma_{i})\times [\sigma_{j-1},\sigma_{j})$.
The claim holds because the probability that {\em no} node of $\pi^{-1}(u_k)$ connects to {\em both} $v_1$ and $v_{k-1}$ is given by $(1-p_{1k}p_{k-1,k})^{|\pi^{-1}(u_k)|}$.  
Thus, the probability of the complementary event is give by~\eqref{eq:probclosure}. 

Next, recall that $\underline a >0$ is the uniform lower bound for the nonzero entries of $A(x)$, for all $x\in \cU^*$, introduced in item~\ref{it:underlinea} of Proposition~\ref{prop:propnijn}. 
Because $x(G_n) = A(x(G_n)) \bfo $, every entry of $x(G_n)$ is bounded below by $\underline a$ as well, so 
\begin{equation}\label{eq:boundsonVi}
|\pi^{-1}(u_i)| \ge \underline a n, \mbox{ for all }  i = 1,\ldots, q.
\end{equation} 
Thus, the expression~\eqref{eq:probclosure} can be  lower bounded by
$$
1 - (1-p_{1k}p_{k-1,k})^{|\pi^{-1}(u_k)|} \geq 1-(1- p^2)^{\underline a n}, 
$$
where $p:= \min \{p_{ij} \mid (u_i,u_j)\in F\} > 0$. 
Note that the  right-hand-side of the above equation converges to $1$ as $n\to\infty$, so $\ce_{11}$ is true \as.

Let $n':= n - |C_1|$. Conditioning on the event $\ce_{11}$, we let $G'_{n'}$ be the subgraph of $G_n$ induced by the nodes  not in $C'_1$. 
Similarly as above, we have that there is a cycle $C'_2$, with $\pi(C'_2)=\pi(C_2)$, in $G'_{n'}$ \as~(note that $n\to\infty$ implies $n'\to\infty$). Iterating this argument for finitely many steps, we have that $\ce_1$ is true \as.
\end{proof}

In the sequel, we condition on the event $\ce_1$ and let $D'_1,\ldots, D'_m$ be the directed cycles in $\vec G_n$ corresponding to $D_1,\ldots D_m$ in $H$ of $\vec G^s_n$.  

\vspace{.2cm}

\noindent
{\em Part 2: On 2-cycles.} 
Let $n':=n- \sum^m_{i = 1} |C'_i|$, and $G'_{n'}$ be the subgraph of $G_n$ induced by the nodes that do {\em not} belong to any of the above cycles $C'_i$, and $G'^s_{n'}$ be its saturation. 
By removing the cycles $D_i$ out of $H$,  
we obtain a Hamiltonian decomposition $H'$ for $\vec G'^s_{n'}$, which is comprised  only of 2-cycles. It now suffices to show that $H'$ appears, up to relabeling, in $\vec G'_{n'}$ \as. 

Let $V_{ij} \subset \pi^{-1}(u_i)$ be the set of nodes paired  to nodes in $\pi^{-1}(u_j)$ by $H'$ in $\vec G'^s_{n'}$.  
Since $H'$ is a Hamiltonian decomposition, $\pi^{-1}(u_i)$ can be expressed as the disjoint union of the $V_{ij}$'s, for $u_j$ such that $(u_i,u_j) \in F$. 
By items~\ref{it:lowerboundselfloop} and~\ref{it:lowerboundedges} of Proposition~\ref{prop:Aham}, the cardinality of $V_{ij}$, which is the same as the cardinality of $V_{ji}$, is at least $m_{ij}:= n\min\{a_{ij}, a_{ji}\}$. 
Because the nonzero $a_{ij}$'s are bounded below by $\underline a$ by item~\ref{it:underlinea} of Proposition~\ref{prop:propnijn}, we have that $m_{ij}  \ge \underline a n$. 

Suppose that $u_i$ has a self-loop; then, we let $K_i$ be the subgraph  of $G'_{n'}$ induced by the nodes $V_{ii}$.  The graph $K_i$ is an Erd\H os-R\'enyi graph with parameter $p_{ii} >0$ and, by item~\ref{it:lowerboundselfloop} of Proposition~\ref{prop:Aham}, $|V_{ii}| = m_{ii}$ is an even integer. Since $n\to\infty$ implies that $m_{ii}\to\infty$,  $K_i$ has a perfect matching \as. This holds because one can split the node set $V_{ii}$ into two disjoint subsets of equal cardinality and apply Lemma~\ref{lem:RGfacts}. 
In other words, it is \as.~that there are  $m_{ii}/2$, for $i = 1,\ldots, q$, disjoint 2-cycles in $\vec G'_{n'}$ pairing nodes in $\pi^{-1}(u_i)$. 

Suppose that $(u_i,u_j)$ is an edge between two distinct nodes; then, we 
let $B_{ij}:=B(|V_{ij}|, |V_{ji}|, p_{ij})$ be the bipartite graph in $G'_{n'}$ induced by $V_{ij} \cup V_{ji}$ (recall from above that $|V_{ij}| = |V_{ji}|$). Let $\ce_{ij}$ be the event that $B_{ij}$ has a perfect matching. 
Since $m_{ij}\ge \underline a n$, by Lemma~\ref{lem:RGfacts}, the event $\ce_{ij}$ holds \as~and, hence, it is \as~that there are $|V_{ij}|$ disjoint 2-cycles in $\vec G'_{n'}$ pairing nodes from $V_{ij}$ to $V_{ji}$. 

Since there are finitely many edges in $S$, by the above arguments, we conclude that $H'$ appears in $\vec G'_{n'}$ \as. This completes the proof. \hfill{\qed}

\section{Conclusions}
Hamiltonian decompositions underlie a wide range of structural properties of control systems, such as stability and ensemble controllability. We say that a graphon $W$ satisfies the $H$-property if graphs $G_n \sim W$ have a Hamiltonian decomposition almost surely. In a series of papers, of which this is the second, we exhibited necessary and  sufficient conditions for the $H$-property to hold for the class of step-graphons. These conditions are geometric and revealed the fact that $H$-property depends only on concentration vector and skeleton graph of $W$. When these two objects are given, one can reconstruct a step-graphon modulo the exact value of $W$ on its support, thus giving rise to an equivalence relation on the space of step-graphons. We  showed that the  $H$-property is essentially a ``zero-one''  property of the equivalence classes. The case of general graphons will be addressed in future work.

\bibliographystyle{ieeetr}
\bibliography{refs}

\appendix

\section{Analysis and Proof of Proposition~\ref{prop:invarianceofskeleton}}\label{app:prop1}
We first have some preliminaries about {\em refinements} of partitions: given a partition $\sigma$, a {\em refinement} $\sigma'$ of $\sigma$, denoted by $\sigma\prec \sigma'$, is any sequence that has $\sigma$ as a proper subsequence. For example, $\sigma' = (0,1/2,3/4,1)$ is a refinement of $\sigma = (0,1/2,1)$. 
Given a step-graphon $W$, if $\sigma$ is a partition for $W$, then so is $\sigma'$.

We say that $\sigma'$ is a {\em one-step refinement} of $\sigma$ if it is a refinement with $|\sigma'|=|\sigma|+1$. Any refinement of $\sigma$ can be obtained by iterating one-step refinements. To fix ideas,  and without loss of generality, we consider the refinement of  $\sigma=(\sigma_0,\ldots, \sigma_q,\sigma_*)$ to $\sigma'=(\sigma_0,\ldots, \sigma_q,\sigma_{q+1}, \sigma_*)$ with $\sigma_{q} < \sigma_{q+1} < \sigma_*$. 
If $S = (U,F)$, then $S'=(U',F')$, the skeleton graph of $W$ for $\sigma'$, is given by
\begin{equation}\label{eq:splitnode}
\left\{
\begin{aligned}    
    U' & = U  \cup  \{ u_{q+1}\}, \\
    F' & = F \cup  \{ (u_i, u_{q+1}) \mid (u_i,u_q)\in F \} 
    \cup \{(u_{q+1}, u_{q+1}) \mbox{ if } (u_q, u_q)\in F\}.
\end{aligned}
\right.
\end{equation}
In essence, the node $u_{q+1}$ is a copy of the node $u_q$. If there is a  loop $(u_q,u_q)$ in $F$, then $u_q$ and $u_{q+1}$ are also connected and each has a self-loop. See Figure~\ref{fig:split} for illustration. 
We say that a one-step refinement {\em splits a node} (here, $u_q$).

\begin{figure}
    \centering
    \subfloat[\label{sfig1:beforesplit}][Graph $S$]{
\includegraphics{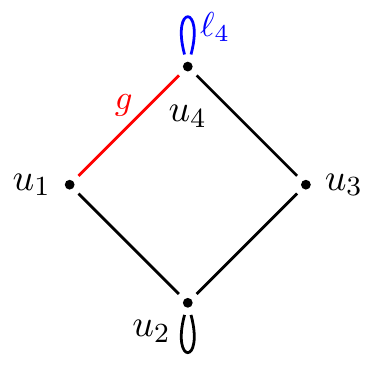}
}
\qquad
\subfloat[\label{sfig1:aftersplit}][Graph $S'$]{
   \includegraphics{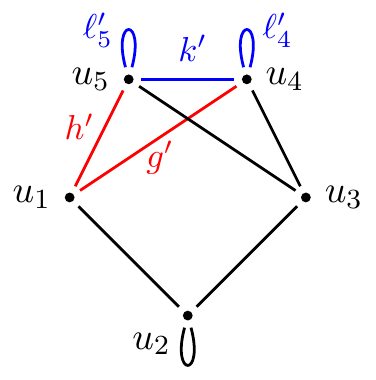}
}
\caption{ The graph $S'$ on the right is obtained from the left $S$ via a one-step refinement. The node $u_4$ in $S$ is split into $u_4$ and $u_5$. 
Because  $g=(u_1,u_4)$ is an edge in $S$, there exist two edges $g'=(u_1,u_4)$ and $h'=(u_1,u_5)$ in $S'$. Because $u_4$ has a self-loop $\ell_4$ in $S$, both $u_4$ and $u_5$ have self-loops in $S'$, denoted by $\ell'_4$ and $\ell'_5$, respectively. In addition, we have the edge $k'=(u_4,u_5)$ in $S'$.}
    \label{fig:split}
\end{figure}

We now prove Proposition~\ref{prop:invarianceofskeleton}: 

\begin{proof}[Proof of Proposition~\ref{prop:invarianceofskeleton}] 
Let $\sigma$ and $\sigma'$ be as given in the statement of the proposition. 
It should be clear that there exists another partition $\sigma''$ which is a refinement of both $\sigma'$ and $\sigma$ and that $\sigma''$ can be obtained via a sequence of one-step refinements starting with either $\sigma'$ or $\sigma$. Thus, combining the arguments at the beginning of the section, we can assume, without loss of generality, that $\sigma'$ is a {\em one-step} refinement of $\sigma$ obtained by splitting the node $u_1 \in U$. 

Let $x^*$ and $x'^*$ be the concentration vectors for $\sigma$ and $\sigma'$, $S$ and $S'$ be the corresponding skeleton graphs, and $Z$ and $Z'$ be the corresponding incidence matrices. 
Note that $Z'$ has one more row than $Z$ does due to the addition of the new node $u_{q+1}$; here, we let the last row of $Z'$ correspond to that node. It should be clear that $Z'$ contains $Z$ as a submatrix. 
For clarity of presentation, we use $f$ (resp. $f'$) to denote edges of $S$ (resp. $S'$). Since the graph $S$ can be realized as a subgraph of $S'$ in a natural way, we will write on occasion $f'\in F$ if $f'$ is an edge of $S$.  

We now prove the invariance of each item listed in the statement of Proposition~\ref{prop:invarianceofskeleton} under one-step refinements.
The proofs of the first two items are direct consequence of the definition of one-step refinement.

\mab{Proof for item 1):} If $S$ is connected, then  from~\eqref{eq:splitnode} we obtain that there exists a path from any node $u_i \in F$ to the new node $u_{q+1}$, so $S'$ is also connected. Reciprocally, assume that $S$ has at least two connected components. 
Then, the node $u_{q+1}$ obtained by splitting $u_q$ will only be connected to nodes in the same component as $u_q$ by definition of $F'$. 

\mab{Proof for item 2):} If $S$ has an odd cycle, then so does $S'$ by~\eqref{eq:splitnode}. Reciprocally, we assume that $S$ is lacking an odd cycle. 
We show that $S'$ has no odd cycle. Suppose, to the contrary, that it does.  The cycle must then contain the node $u_{q+1}$. Replacing $u_{q+1}$ with $u_q$ yields a closed walk of odd length in $S$.  Since a closed walk can be decomposed edge-wise into a union of cycles and since the length of the walk is the sum of the lengths of the constituent cycles, there must exist an odd cycle in $S$, which is a contradiction.

\mab{Proof for item 3):} We prove each direction of the statement separately:

\xc{\bf Part 1: $x\in \cX(S) \Rightarrow x'\in \cX(S')$ ($x\in \rint \cX(S) \Rightarrow x'\in \rint \cX(S')$):} 
For ease of presentation, we let $z_f$ (resp. $z'_{f'}$) be the edge of $Z$ (resp. $Z'$) corresponding to the element $f\in F$ (resp. $f'\in F'$), and $z_{f,i}$ be the $i$th entry of $z_f$. 
Because $\cX(S)$ is the convex hull of the columns of $Z$,  there exist coefficients $c_f\geq 0$, for $f\in F$, such that $x= \sum_{f\in F} c_f z_f$. 
If, further, $x\in \rint \cX(S)$, then these coefficients can be chosen to be strictly positive. 
We will use $c_f$ to construct $c'_{f'} \ge 0$, for $f' \in F'$, such that 
\begin{equation}\label{eq:x'c'f'z'}
x'=\sum_{f' \in F'} c'_{f'} z'_{f'}
\end{equation}
and show that $x'\in \rint \cX(S')$ if $x\in \rint \cX(S)$.  

To proceed, let $F_{u_q}$ be the set of edges incident to node $u_q$ in $S$. Similarly, let $F'_{u_q}$ and $F'_{u_{q+1}}$ be the sets of edges incident to $u_q$ and $u_{q+1}$ in $S'$, respectively. 
The coefficients $c'_{f'}$ are defined as follows:
\begin{enumerate}
    \item[a)] If $f' \notin F'_{u_q} \cup F'_{u_{q+1}}$, then $f'\in F$. Let $c'_{f'}:=c_{f'}$. 
    \item[b)] If $f'\in F'_{u_q}$ and $f' \neq (u_q,u_{q+1})$, then $f'\in F$. Let $c'_{f'} := \frac{\sigma_{q+1} - \sigma_q}{\sigma_* -\sigma_q} c_{f'}$.
    
    \item[c)] If $f'= (u_i, u_{q+1})$ and $u_i \neq u_q$, then we pick the $f \in F$ such that 
    $$f =\begin{cases} (u_i, u_q) & \mbox{ if }u_i \neq u_{q+1}, \\
    (u_q,u_q) & \mbox{ if } u_i = u_{q+1}.
    \end{cases}
    $$ 
   Let $c'_{f'} := \frac{\sigma_* - \sigma_{q+1}}{\sigma_* -\sigma_q} c_{f}$. 
    
    \item[d)] If $f' = (u_q,u_{q+1})$, then let $c'_{f'}:=0$.
\end{enumerate}


With the coefficients as above, we prove entry-wise that~\eqref{eq:x'c'f'z'} holds. 
First, note that because we obtain $S'$ by splitting the last node $u_q$ of $S$, the $i$th entry of $x'$, for $1\leq i\leq q-1$, is equal to $x_i$, so $x'_{i}=x_i = (\sigma_{i} - \sigma_{i-1})$. For the $i$th entry of the right hand side of~\eqref{eq:x'c'f'z'}, we consider two cases: 
\vspace{.1cm}

\noindent
{\em Case 1: $u_{i}$ is not incident to $u_q$ in $S$.} In this case, $u_i$ is not incident to either $u_q$ or $u_{q+1}$ in $S'$. Consequently, $F'_{u_{i}} = F_{u_{i}}$ and $z'_{f,i} = z_{f,i}$ for all $f\in F_{u_i}$. Furthermore, by item (a), $c'_{f} = c_f$ for any $f\in F_{u_{i}}$.  
Thus, the $i$th entry of the right hand side of~\eqref{eq:x'c'f'z'} is given by
$$
\sum_{f'\in F'_{u_{i}}} c'_{f'} z'_{f',i} = \sum_{f\in F_{u_{i}}}  c_f z_{f,i} = x_{i}  = \sigma_{i} - \sigma_{i-1}.
$$
\vspace{.1cm}

\noindent
{\em Case 2: $u_{i}$ is incident to $u_q$ in $S$.} In this case, $u_i$ is incident to both $u_q$ and $u_{q+1}$ in~$S'$. 
Let $g':= (u_i,u_q)$ and $h':= (u_i,u_{q+1})$ be the corresponding edges in~$S'$, see Fig.~\ref{fig:split} for an illustration. 
Then, the $i$th entry of the right hand side of~\eqref{eq:x'c'f'z'} is given by
\begin{equation}\label{eq:ithentryofx'}
\sum_{f'\in F'_{u_{i}}} c'_{f'} z'_{f',i} = c'_{g'} z'_{g',i} + c'_{h'} z'_{h',i}+ \sum_{f'\in F'_{u_i} - \{g', h'\}} c'_{f'} z'_{f',i}.
\end{equation}
By items (b) and (c), 
$$c'_{g'} = \frac{\sigma_{q+1} - \sigma_q}{\sigma_* - \sigma_q} c_{g'} \quad \mbox{and}  \quad c'_{h'} = \frac{\sigma_* - \sigma_{q+1}}{\sigma_* - \sigma_q} c_{h'}.$$ 
Also, note that  
$$z'_{g',i}=z'_{h',i}= z_{g',i} = \frac{1}{2}.$$
Thus, the sum of the first two terms on the right hand side of~\eqref{eq:ithentryofx'} is $c_{g'} z_{g',i}$. For the last term,  note that $F'_{u_i} - \{g',h'\}=F_{u_i} - \{g'\}$. Also, by item (a) and the fact that $z'_{f,i} = z_{f,i}$ for any $f\in F_{u_i}$,  
$$
\sum_{f'\in F'_{u_i} - \{g',h'\}} c'_{f'} z'_{f',i} = \sum_{f\in F_{u_i} -\{g'\}} c_{f} z_{f,i}. 
$$
Combining the above arguments, we have that the right hand side of~\eqref{eq:ithentryofx'} is given by
$$
\sum_{f\in F_{u_{i}}} c_f z_{f,i} = x_{i} = \sigma_{i} - \sigma_{i-1}.
$$

Next, the $q$th entry of $x'$ is $(\sigma_{q+1} - \sigma_q)$ and the $q$th entry of the right hand side of~\eqref{eq:x'c'f'z'} is 
\begin{equation*}
\begin{aligned}
\sum_{f'\in F'_{u_q}} c'_{f'} z'_{f',q} & = \frac{\sigma_{q+1} - \sigma_q}{\sigma_* -\sigma_q}\sum_{f\in F_{u_q}}  c_f z_{f,q}\\ & = \frac{\sigma_{q+1} - \sigma_q}{\sigma_* -\sigma_q} x_q =   \sigma_{q+1} -\sigma_q,
\end{aligned}
\end{equation*}
where the first equality follows from the fact that 
$$F'_{u_q} = F_{u_q} \cup \{(u_q,u_{q+1}) \mbox{ if } (u_q,u_q)\in F \},$$ 
items (b) and (d), and the last equality follows from the fact that $x_q = \sigma_{*} - \sigma_q$.

The last (i.e., $(q+1)$th) entry of $x'$ is $(\sigma_* - \sigma_{q+1})$. The last entry of the right hand side of~\eqref{eq:x'c'f'z'} is given by 
\begin{equation*}
\begin{aligned}
\sum_{f'\in F'_{u_{q+1}}} c'_{f'} z'_{f',q+1} & = \frac{\sigma_* - \sigma_{q+1}}{\sigma_* -\sigma_q}\sum_{f\in F_{u_q}} c_f z_{f,q} \\
& = \frac{\sigma_* - \sigma_{q+1}}{\sigma_* -\sigma_q} x_q = \sigma_* -\sigma_{q+1},
\end{aligned}
\end{equation*}
where the first equality follows from item (c) above.  
We have thus shown that~\eqref{eq:x'c'f'z'} holds. In particular, since $c'_{f'}$ are nonnegative by construction,~\eqref{eq:x'c'f'z'} implies that $x'\in \cX(S')$.

It now remains to show that if $x \in \rint \cX(S)$, then $x' \in \rint \cX(S')$. Assuming $x \in \rint \cX(S)$,  if $u_q$ does not have a self-loop in $S$, then the edge $(u_q,u_{q+1})$ does not exist in $S'$, so by items (a), (b), and (c), all coefficients $c'_{f'}$ are positive, which implies that $x' \in \rint \cX(S')$.

We now assume that $u_q$ has a self-loop in $S$. 
Then, $k':= (u_q,u_{q+1})$ is an edge in $S'$ (see Fig.~\ref{fig:split} for an illustration), and thus $c'_{k'}=0$ per item (d) above. 
In this case, both $u_q$ and $u_{q+1}$ have self-loops in $S'$. Denote these two self-loops by $\ell'_{q} := (u_q,u_q)$ and $\ell'_{q+1} := (u_{q+1}, u_{q+1})$. 
By~\eqref{eq:defZS}, we have that 
$$z'_{k'}= \frac{1}{2} (z'_{\ell'_{q}} +  z'_{\ell'_{q+1}}).$$ 
Since $c'_{\ell'_{q}}$ and $c'_{\ell'_{q+1}}$ are positive, there exists an $\epsilon> 0$ such that $\epsilon < c'_{\ell'_{q}}$ and $\epsilon < c'_{\ell'_{q+1}}$. It then follows that
\begin{equation}\label{eq:manyprimes}
c'_{\ell'_{q}} z'_{\ell'_{q}} + c'_{\ell'_{q+1}} z'_{\ell'_{q+1}} = 2\epsilon z'_{k'} + (c'_{\ell'_{q}} - \epsilon) z'_{\ell'_{q}} + (c'_{\ell'_{q+1}} - \epsilon) z'_{\ell'_{q+1}}.
\end{equation}
Plugging in~\eqref{eq:x'c'f'z'} the relation~\eqref{eq:manyprimes} shows that $x'$ can be written as a convex combination of the $z'_{f'}$, for $f'\in F'$, with all {\em positive} coefficients,  and thus  $x' \in \rint \cX(S')$.

\xc{\bf Part 2: $x'\in \cX(S')\Rightarrow  x\in \cX(S)$ ($x'\in \rint \cX(S')\Rightarrow  x\in \rint \cX(S)$). } 
Because $x'\in \cX(S')$ (resp. $x'\in \rint \cX(S')$), we can write $x' = \sum_{f'\in F'}c'_{f'} z'_{f'}$, with $c'_{f'} \geq 0$ (resp. $c'_{f'} > 0$), for all $f'\in F'$. We will use $c'_{f'}$ to construct $c_{f}$, for $f\in F$, so that 
\begin{equation}\label{eq:xczf}
x = \sum_{f\in F} c_f z_f.
\end{equation}
To this end, we define $c_f$ as follows:
\begin{enumerate}
    \item[e)] If $f$ is not incident to $u_q$ in $S$, 
    then let $c_f := c'_f$. 
    \item[f)] If $f = (u_i,u_q)$ and $u_i\neq u_q$, then $g':=(u_i,u_q)$ and $h':=(u_i,u_{q+1})$ are edges in $S'$, and let $c_f := c'_{g'} + c'_{h'}$.
    \item[g)] If $f = (u_q,u_q)$, then  $k':=(u_q,u_{q+1})$, $\ell'_{q} := (u_q,u_q)$, and $\ell'_{q+1} := (u_{q+1},u_{q+1})$ are edges in $S'$, and let $  c_f := c'_{k'}+ c'_{\ell'_{q}}+c'_{\ell'_{q+1}}$. 
\end{enumerate}
Note that all the coefficients $c_f$, for $f\in F$, defined above are nonnegative. Further, if all the $c'_{f'}$ are positive, i.e., $x'\in \rint \cX(S')$, then the $c_{f}$ are positive as well, which implies $x\in \rint \cX(S)$ provided that~\eqref{eq:xczf} holds.

We now show that the coefficients given above are such that~\eqref{eq:xczf} indeed holds. We do so by checking that~\eqref{eq:xczf}  holds for each entry.

For the $i$th entry, with $1 \leq i <q$, the left hand side of~\eqref{eq:xczf} is $x_i = (\sigma_i - \sigma_{i-1})$. For the right-hand side, if $(u_i,u_q)$ is an edge in $S$, then $g'$ and $h'$, as defined item~(f), are two edges in $S'$ and, consequently,  $F'_{u_i} = F_{u_i} \cup \{h'\}$.  
Note that $z_{f,i} = z'_{f,i}$ for all $f\in F_{u_i}$ and $$z_{g',i} = z'_{g',i} = z'_{h',i} = \frac{1}{2}.$$
Thus, by items (e) and (f), we have that
\begin{align*}
\sum_{f \in F} c_f z_{f,i} & =  c_{g'} z_{g',i} + \sum_{f \in F_{u_i}-\{g'\} }c_{f}z_{f,i} \\
& = c'_{g'} z'_{g',i} + c'_{h'} z'_{h',i} + \sum_{f'\in F'_{u_i} - \{g', h'\}} c'_{f'} z'_{f',i} \\
& = \sum_{f'\in F'_{u_i}} c'_{f'} z'_{f',i} = x'_{i} = (\sigma_{i} - \sigma_{i-1}).
\end{align*}

Finally, for the last entry, i.e., the $q$th entry, the left hand side of~\eqref{eq:xczf} is $x_q= \sigma_*-\sigma_q$. 
For the right hand side of~\eqref{eq:xczf}, we let $\ell_{q} := (u_q,u_q)$ be the loop on $u_q$ (if it exists in $S$) and thus have that
\begin{equation}\label{eq:firstentry}
\sum_{f\in F_{u_q}} c_f z_{f,q}  = c_{\ell_q} z_{\ell_{q},q} + \sum_{f\in F_{u_q} - \{\ell_{q}\}} c_f z_{f,q}. 
\end{equation}
Let $k'$, $\ell'_q$, and $\ell'_{q+1}$
be the three edges in $S'$ as defined in item~(g).  
Note that 
$$z_{\ell_{q},q} = z'_{\ell'_{q},q}=z'_{\ell'_{q+1},q+1}=2z'_{k',q} = 2z'_{k',q+1} = 1.$$
For the first term of~\eqref{eq:firstentry}, using item (g) and the above relations, we  obtain 
\begin{equation}\label{eq:cf11zf11}
c_{\ell_{q}} z_{\ell_{q},q} = c'_{\ell'_{q}} z'_{\ell'_{q},q} + c'_{\ell'_{q+1}}z'_{\ell'_{q+1},q+1} +c'_{k'}z'_{k',q} + c'_{k'}z'_{k',q+1}.
\end{equation}
For each addend in the second term of~\eqref{eq:firstentry}, the edge $g  = (u_i,u_q)$ in $S$, for some $u_i\neq u_q$,  has two corresponding edges in $S'$, namely  
$g' = (u_i,u_q)$ and 
$h' = (u_i,u_{q+1})$. 
Note that
$$
z_{g,q} = z'_{g',q} = z'_{h',q+1} = \frac{1}{2}.
$$
Then, by item (f), 
\begin{equation}\label{eq:cf1izf1i}
c_{g} z_{g,q} = c'_{g'} z'_{g',q} + c'_{h'}z'_{h',q+1}.  
\end{equation}
Combining~\eqref{eq:cf11zf11} and~\eqref{eq:cf1izf1i}, we obtain that
\begin{align*}
\sum_{f\in F_{u_q}} c_f z_{f,1} & = \sum_{f'\in F'_{u_q}} c'_{f'} z'_{f',q}  + 
\sum_{f'\in F'_{u_{q+1}}} c'_{f'} z'_{f',q+1}\notag \\ 
& = x'_q + x'_{q+1} = (\sigma_{q+1} - \sigma_q) + (\sigma_* - \sigma_{q+1}) \\
&= \sigma_* - \sigma_q.  \notag
\end{align*}
This concludes the proof.
\end{proof}

\section{Proof of Lemma~\ref{lem:nijnprelim}}\label{app:nijnprelim}
\xcb{Proof of item~1:} From the definitions of $\tau'_i(x)$, we have that 
\begin{equation}\label{eq:boundx0x0p} 
-\bfo \le n (\tau_0(x) - \tau'_0(x))  = n (\tau'_1(x) - \tau_1(x))\le \bfo.
\end{equation}
If $\tau_{0,i}(x) = 0$, then $\tau'_{0,i}(x) = 0$, where $\tau_{0,i}(x)$ is the $i$th entry of the vector $\tau_0(x)$. Otherwise, by the definition of the incidence matrix~\eqref{eq:defZS} and by \eqref{eq:defx0x1} and~\eqref{eq:defalphalowerbound}, we have that $\tau_{0,i}(x)\ge \alpha$. For the latter case, by~\eqref{eq:boundx0x0p} and the hypothesis on $n$ in the statement of Proposition~\ref{prop:propnijn}, we have that 
\begin{equation}\label{eq:lowerboundx0prime}
\tau'_{0,i}(x) \ge \tau_{0,i}(x) - \frac{1}{n} \ge (1 -   \epsilon/8) \alpha > 0.
\end{equation}
It then follows that
\begin{equation}\label{eq:suppx0'x0}
\supp \tau'_0(x) = \supp \tau_0(x). 
\end{equation}

Similarly, for $\tau_1(x)$, using~\eqref{eq:defZS},~\eqref{eq:defx0x1},  and~\eqref{eq:defalphalowerbound}, we have that $\tau_1(x)\geq \frac{1}{2}\alpha\bfo$. Then, again,  by~\eqref{eq:boundx0x0p} and the hypothesis on $n$,  
$$
\tau'_1(x) \ge \tau_1(x) - \frac{1}{n}\bfo \ge  \frac{1}{2} (1 - \epsilon /4 )  \alpha \bfo > 0, 
$$
from which we conclude that $$\supp \tau'_1(x) = \supp \tau_1(x) = \{1,\ldots,q\}.$$ This concludes the proof of the first item.

\xc{Proof of item~2:}
By~\eqref{eq:lowerboundx0prime},~\eqref{eq:suppx0'x0}, and Lemma~\ref{lem:supportx0x1}, the ratio $n'_0/n$ is uniformly lower bounded by 
\begin{equation}\label{eq:lowerboundn0n}
\frac{n'_0}{n} =  \|\tau'_0(x)\|_1 \ge (1 -  \epsilon/8) \alpha |\supp \tau'_0(x)| = (1 -  \epsilon/8) \alpha |F_0|.
\end{equation}
To obtain a lower bound for $n_1'/n$, we let 
$$\eta:= \tau_1(x) - \tau'_1(x).$$ 
From~\eqref{eq:boundx0x0p} and the hypothesis on $n$, we have that 
$$
\|\eta\|_1\le \frac{q}{n} \le \frac{1}{8} q \alpha\epsilon  \le \frac{1}{8}q\alpha.
$$ 
It then follows that
\begin{equation}\label{eq:n1overn}
\frac{n_1'}{n} \ge  \|\tau_1(x)\|_1 - \|\eta\|_1 \ge \frac{1}{2}q\alpha - \frac{1}{8} q \alpha = \frac{3}{8} q\alpha. 
\end{equation}
This concludes the proof of the second item.

\xc{Proof of item~3:}
By~\eqref{eq:boundx0x0p} and the hypothesis on $n$, $\eta$ as introduced above satisfies
$$\|\eta\|_\infty \leq 1/n < \alpha\epsilon/8 < \epsilon.$$ 
Because $\bar \tau_1'(x) = \overline{\tau_1(x) + \eta}$ and $\|\eta\|_\infty < \epsilon$, we have that $\bar \tau'_1(x)\in \cV$ by Lemma~\ref{lem:xbx1epsilon}. \hfill{\qed}

\section{Proof of Lemma~\ref{lem:redH}}\label{app:prop3}

The proof is carried out by induction on $n'$. If $n'=0$, then $G'$ is the empty graph and there is nothing to prove. For the inductive step, we  set $n'>0$ and  assume that the lemma holds for all $n'' < n'$, and prove it for $n'$. 

To proceed, we use $A'$ to turn $\vec S$ into a weighted digraph on $q$ nodes: we assign to edge $u_iu_j$ the weight $a'_{ij}$.  
Then, $\vec S$ is a balanced graph, i.e.,  
\begin{equation}\label{eq:balanceequation}
\sum_{u_j\in N_-(u_i)}a'_{ij} = \sum_{u_j\in N_+(u_i)}a'_{ji}, \quad \forall i = 1,\ldots, q,
\end{equation}
where we recall $N_-(u_i)$ and $N_+(u_i)$ are the sets of out-neighbors and in-neighbors of $u_i$, respectively, in $\vec S$.

Let $\vec S'$ be the subgraph of $\vec S$ induced by the nodes incident to the edges with nonzero weights. 
Then,  $\vec S'$ has at least one cycle. To see this, note that if $\vec S'$ is acyclic, then by relabeling the nodes, the matrix $A'$ is upper-triangular and, from the hypothesis, $\diag A' = 0$.  
It follows that the only {\em nonnegative} solution $\{a'_{ij}\}$ to~\eqref{eq:balanceequation} is that all the $a'_{ij}$ are zero, which contradicts the fact that $A'$ is nonzero.   

Since $\vec S'$ is a subgraph of $\vec S$, any cycle of $\vec S'$ is also a cycle of $\vec S$; denote such cycle  by $D_S:=u_{i_1}\ldots u_{i_k}u_{i_1}$. By construction, the weights of the edges in the cycle are positive. It thus follows from $A'\bfo =x'$ that the entries $x'_{i_j}$, for $j = 1,\ldots, k$, are positive; together with the fact that  $G' = M(n'x',S)$, it implies that the sets $\pi^{-1}(u_{i_j}) \in G'$, for $j = 1,\ldots, k$, are non-empty.  
We next pick a node $v_{j}$ from each $\pi^{-1}(u_{i_j})$. 
Since the nodes $u_{i_1},\ldots, u_{i_k}$ are pairwise distinct, so are the nodes $v_1,\ldots, v_k$.
Also, since $G'$ is a complete $S$-multipartite graph, $D_G:=v_{1}\cdots v_{k}v_{1}$ is a cycle in $\vec G'$. Moreover, by construction, $\pi(D_G) = D_S$ and, hence, $D_G$ is simple.  

We let $G''$ be the graph obtained by removing from $G'$ the $k$ nodes $v_1,\ldots,v_k$, and the edges incident to them. Then, $G''$ is a complete $S$-multipartite graph on $n'':=n'-k$ nodes. 
Define 
$$
x'' :=\frac{1}{n''}(n'x' - \sum^k_{j = 1} e_{i_j}),
$$
where $\{e_{1},\ldots, e_q\}$ is the canonical basis of $\R^q$.  
Note that $x''\geq 0$; indeed, $n'x'$ is integer valued and $x'_{i_j}$, for $j = 1,\ldots,k$, are positive. 
We can then write $G''=M(n''x'',S)$.  
Correspondingly, we define a $q\times q$ matrix $A''$ as follows: 
$$
A'' := \frac{1}{n''} (n'A' - \sum^{k-1}_{j = 1} e_{i_j} e_{i_{j+1}}^\top - e_{i_k}e_{i_1}^\top).
$$
In words, to obtain $n''A''$, we decrease the $ij$th entry of $n'A'$, which is a positive integer, by one when the edge $u_i u_j$ is used in the cycle $D_S$ and keep the other entries unchanged.  

By construction, we have that $A'' \in \cA(S)$, with ${A''}\bfo=x''$, and $n'' A''$ is integer-valued. Because $n'' < n'$, we can appeal to the induction hypothesis and exhibit a Hamiltonian decomposition $H''$ of $\vec G''$ such that $\rho(H'')=A''$ and every cycle in $H''$ is simple. It is clear that adding the simple cycle $D_G$ to $H''$ yields a Hamiltonian decomposition $H'$ of $\vec G'$ with desired properties. This completes the proof.  \hfill{\qed}

\section{Proof of Lemma~\ref{lem:RGfacts}}\label{app:bipartite}
\xcb{1. Proof that $B(n_1,n_2,p)$ has a left-perfect matching a.s.:}
The proof of this part relies on the following statement, which is a consequence of a stronger result of Erd\H os and R\'enyi~\cite{erdos1964random}: For $p \in (0,1)$ a constant, the random bipartite graph $B(m,m,p)$ contains a perfect matching \as. Now, without loss of generality, we assume that $n_1\leq n_2$ and let $B(n_1,n_1,p)$ be a subgraph of $B(n_1,n_2,p)$. Since $n_2/n_1$ is bounded above by a constant $\kappa$, $n\to\infty$ implies that $n_1\to\infty$. Since $B(n_1,n_1,p)$ has a (left-)perfect matching \as, so does $B(n_1,n_2,p)$.

\xc{2. Proof that $B(n_1,n_2,p)$ is connected a.s.:} It is well known (see, e.g.,~\cite[Exercise 4.3.7]{frieze2016introduction}) that $B(m,m,p)$ is connected \as. We now extend the result to the general case where $n_1$ is not necessarily equal to $n_2$. Again, we can assume without loss of generality that $n_1\leq n_2$. 
Let $V_L=\{\alpha_1,\ldots,\alpha_{n_1}\}$ and $V_R=\{\beta_1,\ldots, \beta_{n_2}\}$ be the left- and right-node sets of $B$. 
Because $n_2/n_1\leq \kappa$, we can choose $\kappa$ subsets $V_{R,i} \subseteq V_R$, so that $|V_{R,i}|=n_1$ and  $\cup^\kappa_{i=1} V_{R,i}=V_R$.

Denote by $\mathcal{E}_i$ the event that the subgraph $B_i$ of $B:=B(n_1,n_2,p)$ induced by $V_{L}$ and $V_{R,i}$ is disconnected and by $\mathcal{E}$ the event that $B$ is disconnected. Note that if every $B_i$ is connected, then so is $B$. 
Conversely, we have that $\mathcal{E} \subseteq \cup_{i=1}^\kappa \ce_i$. 

Note that $B_i=(n_1,n_1,p)$ and, as argued above, $n_1\to\infty$ as $n\to\infty$. Since $B_i$ is connected \as, $\lim_{n \to \infty} \bP(\ce_i)=0$ and, hence, $$\lim_{n\to\infty}\bP(\ce) \le \lim_{n\to\infty}\sum^\kappa_{i = 1} \bP(\ce_i) = 0.$$ 
This completes the proof.\hfill \qed

\end{document}